\documentclass[12pt,bezier]{article}
\usepackage{times}
\usepackage{booktabs}
\usepackage{pifont}
\usepackage{floatrow}
\usepackage{caption}
\usepackage{mathrsfs}
\usepackage[fleqn]{amsmath}
\usepackage{amsfonts,amsthm,amssymb,mathrsfs,bbding}
\usepackage{txfonts}
\usepackage{graphics,multicol}
\usepackage{graphicx}
\usepackage{color}
\usepackage{amssymb}
\usepackage{caption}
\usepackage{cite}
\usepackage{latexsym,bm}
\usepackage{indentfirst}
\usepackage{color}
\usepackage[colorlinks=true,anchorcolor=blue,filecolor=blue,linkcolor=blue,urlcolor=blue,citecolor=blue]{hyperref}
\usepackage{extarrows}
\usepackage{cite}
\usepackage{latexsym,bm}
\usepackage{mathtools}
\evensidemargin=\oddsidemargin\topmargin=-1.5cm

\newtheorem{thm}{Theorem}[section]

\newtheorem{prob}{Problem}[section]
\newtheorem{claim}{Claim}

\newtheorem{lem}{Lemma}[section]
\newtheorem{cor}{Corollary}[section]

\theoremstyle{definition}

\addtocounter{section}{0}

\begin{document}
\title{Extremal results for graphs with binding number strictly less than $1/r$ \footnote{Supported by National Natural Science Foundation of China
{(No. 12371361)} and Natural Science Foundation of Henan Province {(Nos. 242300421045 and 252300420303)}.}}
\author{\bf Ruifang Liu$^{a}$, {\bf Hongyu Chen$^{a}$}, {\bf Ao Fan$^{a}$}\thanks{Corresponding author.
E-mail addresses: rfliu@zzu.edu.cn (R. Liu), isik104@163.com (H. Chen),
fanaozzu@163.com (A. Fan).}\\
{\footnotesize $^a$ School of Mathematics and Statistics, Zhengzhou University,
Zhengzhou, Henan 450001, China}}
\date{}

\maketitle
{\flushleft\large\bf Abstract}
The binding number $b(G)$ of a graph, introduced by Woodall [J. Combin. Theory, Ser. B, 1973],
is a central topic of both structural and extremal graph theory.
It is closely related to fundamental combinatorial and structural properties of graphs.

The graphs with $b(G)\geq1$ exhibit strong expansion properties and a highly
connected global structure. In contrast, the structure for graphs with $b(G)<1$ remains far less well understood.
Kane et al. [J. Graph Theory, 1981] proved that if $b(G)<1$, then every binding set
of $G$ is independent. Goddard and Swart [Quaest. Math., 1990] showed that if $b(G)\leq1$, then the toughness $\tau(G)\leq b(G).$
This makes it particularly interesting to investigate extremal problems for graphs with \(b(G)<1\). For any integer $r\geq1,$
we completely characterize the unique extremal graph that maximizes the size (spectral radius) among all graphs of order $n$ satisfying $b(G)<\frac{1}{r}.$

For any bipartite graph $G=(X,Y)$ on $n$ vertices, it is readily seen that
$b(G)\leq\min\{|X|/|Y|,|Y|/|X|\}\leq1.$ Notably, the complete balanced bipartite graph
$K_{\frac{n}{2}, \frac{n}{2}}$ achieves the maximum size (spectral radius) among all
bipartite graphs with $b(G)=1$. In this paper, we completely determine the extremal graphs maximizing the
size or the spectral radius among all bipartite graphs with
$b(G)<\frac{1}{r}$, where $r\geq1$ is an integer.

\begin{flushleft}
\textbf{Keywords:} Binding number, Size, Spectral radius, Extremal graphs
\end{flushleft}
\textbf{AMS Classification:} 05C50; 05C35

\section{Introduction}
Let $G$ be a simple graph with vertex set $V(G)$ and edge set $E(G)$.
The order and size of $G$ are denoted by $|V(G)|=n$ and $|E(G)|=e(G)$,
respectively. Let $G_1$ and $G_2$ be two vertex-disjoint graphs, and let $G_{1}\cup G_{2}$ denote
their disjoint union.
The join $G_1\vee G_2$ is the graph obtained from $G_{1}\cup G_{2}$ by adding all
possible edges between $V(G_1)$ and $V(G_2)$. For a vertex subset $S\subseteq G,$ let $G[S]$ and $|S|$ be the subgraph of $G$ induced by $S$ and the size of $S$, respectively. For a graph $G$ of order $n,$ its
adjacency matrix is the 0-1 matrix
$A(G)=(a_{ij})_{n\times n},$ where $a_{ij}=1$ if $v_{i}\sim v_{j}$ and $a_{ij}=0$
otherwise. Note that $A(G)$ is a real nonnegative symmetric matrix. So all its
eigenvalues are real and can be
ordered non-increasingly as $\lambda_{1}(G)\geq \lambda_{2}(G) \geq \cdots \geq
\lambda_{n}(G).$ The largest eigenvalue $\rho(G)=\lambda_1(G)$ is called the {\it
spectral radius} of $G.$ For
undefined terminology and notation, one can refer to \cite{Bondy2008, Brouwer2011}.

The binding number of graphs was introduced by Woodall \cite{Woodall1973}. For any
vertex $v\in V(G),$ let $N_G(v)$ denote its neighborhood in $G,$ and for
$S\subseteq V(G),$ define
$N_G(S)=\bigcup_{v\in S}N_G(v).$ The {\it binding number} $b(G)$ of a graph $G$ is
given by
$$b(G)=\min\bigg\{\frac{|N_G(S)|}{|S|}:\emptyset\neq S\subseteq V(G), N_G(S)\neq
V(G)\bigg\}.$$

As a fundamental graph invariant, the binding number possesses wide-ranging
applications in numerous fields, such as graph theory, network science, quantum
sensing, and information processing. A
central problem concerning the binding number lies in characterizing its bounds via
fundamental graph structural parameters, including the degree sequence
\cite{Bauer2011}, minimum degree
\cite{Bauer2012}, and connectivity \cite{Woodall1973}.

The binding number of a graph is closely related to the combinatorial structural
properties of the graph. The graphs with $b(G)\geq1$ exhibit strong expansion
properties and a highly connected
global structure. Some fundamental graph properties that are guaranteed by lower
bounds on $b(G)$ include Hamilton cycle \cite{Woodall1973, Woodall1978}, $k$-factor
\cite{Kano1992, Katerinis1987 },
$k$-extendibility \cite{Chen1995, Robertshaw2002}, $k$-clique \cite{Lyle59}.
In contrast, for graphs with $b(G)<1$, the structure is still not well
characterized. Kane et al.\cite{Kane1981} established the following useful result.
A non-empty subset $S\subseteq V(G)$ is a {\it binding set} if and only if
$b(G)=|N(S)|/|S|$.
\begin{thm}[Kane et al. \cite{Kane1981}]\label{thm1}
Let $G$ be a graph. If $b(G)<1$, then every binding set $S$ is independent.
\end{thm}
The {\it toughness} $\tau(G)=\mathrm{min}\{\frac{|S|}{c(G-S)}: S~\mbox{is a cut set
of vertices in}~G\}$ for $G\ncong K_n,$ which is initially proposed by Chv\'{a}tal
\cite{Chvatal}.
Goddard and Swart \cite{Goddard1990} discovered the relationship between $\tau(G)$
and
$b(G)$ when $b(G)\leq1.$
\begin{thm}[Goddard and Swart \cite{Goddard1990}]\label{thm2}
Let $G$ be a graph. If $b(G)\leq1$, then $b(G)\geq \tau(G).$
\end{thm}

Naturally, one wonders about extremal results of graphs with $b(G)<1.$ So we propose the following intriguing problem.

\begin{prob}\label{pro1}
What is the maximum size (spectral radius) among all graphs of order $n$ satisfying $b(G)<\frac{1}{r},$ where $r\geq1$ is an integer?
\end{prob}

Note that $b(G)=0$ if and only if $G$ contains isolated vertices. It is easy to see that the graph $K_{n-1}\cup K_1$ attains the maximum size (spectral radius) among all graphs with $b(G)=0$. Therefore, for Problem \ref{pro1}, it suffices to restrict our attention to graphs without isolated vertices. Let $\mathcal{G}_{n,r}$ be the family of graphs without isolated vertices of order $n$ such that $b(G)<\frac{1}{r}$ for every $G\in\mathcal{G}_{n,r}$, where $r\ge1$ is an integer.

\begin{thm}\label{thm4}
Let $r\geq1$ and $n\geq r+13$ be two integers, and let $G\in\mathcal{G}_{n,r}.$ Then $$e(G)\leq{n-r-1\choose 2}+r+1$$ with equality if and only if $G\cong K_1\vee(K_{n-r-2}\cup(r+1)K_1).$
\end{thm}

\begin{thm}\label{thm5}
Let $r\geq1$ and $n\geq 2r+15$ be two integers, and let $G\in \mathcal{G}_{n,r}.$  Then $$\rho(G)\leq \rho
(K_1\vee(K_{n-r-2}\cup(r+1)K_1))$$ with equality if and only if $G\cong
K_1\vee(K_{n-r-2}\cup(r+1)K_1).$
\end{thm}

By taking $r=1$ in Theorems \ref{thm4} and \ref{thm5}, one can get the following result directly.

\begin{cor}
[Fan and Lin \cite{Fan2024}]
Let $G$ be a connected graph of order $n\geq 14$ with $b(G)<1$. Each of the
following holds.\\
(i) $e(G)\leq{n-2\choose 2}+2$ with equality if and only if $G\cong
K_1\vee(K_{n-3}\cup 2K_1).$\\
(ii) $\rho(G)\leq\rho(K_1\vee(K_{n-3}\cup2K_1))$ with equality if and only if
$G\cong K_1\vee(K_{n-3}\cup2K_1).$
\end{cor}

The research concerning the binding number of bipartite graphs has not yet been involved in the literature.
 For any bipartite graph $G=(X,Y)$ on $n$ vertices, it is readily seen that
$b(G)\leq\min\{|X|/|Y|,|Y|/|X|\}\leq1.$ Notably, the complete balanced bipartite graph
$K_{\frac{n}{2}, \frac{n}{2}}$ achieves the maximum size (spectral radius) over all
bipartite graphs with $b(G)=1$. This naturally leads to an interesting problem.

\begin{prob}\label{pro2}
What is the maximum size (spectral radius) among all bipartite graphs of order $n$ satisfying $b(G)<\frac{1}{r},$ where $r\geq1$ is an integer?
\end{prob}
Note that $b(G)=0$ if and only if $G$ contains isolated vertices. It is easy to see that the graph $K_{ \big\lfloor\frac{n-1}{2}\big\rfloor, \big\lceil\frac{n-1}{2}\big\rceil}\cup K_1$ attains the maximum size (spectral radius) among all bipartite graphs with $b(G)=0$. Consequently, when addressing Problem \ref{pro2}, it is sufficient to focus on graphs that do not contain isolated vertices. Let $\mathcal{B}_{n,r}$ be the family of $n$-vertex bipartite graphs without isolated vertices satisfying $b(G)<\frac{1}{r}$ for each $G\in \mathcal{B}_{n,r},$ where $r\geq 1$ is an integer.

A connected bipartite graph $G =(X, Y)$ is called a {\it double nested graph} if its
vertex sets admit partitions
$X = X_1\cup X_2\cup\cdots\cup X_h$ and $Y=Y_1\cup Y_2\cup\cdots\cup Y_h,$ where
$X_i$ and $Y_j$ are independent sets, and all vertices in $X_i$ are adjacent
to every vertex in
$\bigcup_{j=1}^{h+1-i} Y_j$ for $1\leq i\leq h.$  In the following figure, large solid circles stand for independent sets, and each line between two large solid circles means that all vertices in one large solid circle are adjacent to all vertices in the other one. Let $|X_i|= p_i$ and $|Y_i|=q_i$ for $i=1,2,\ldots,h.$ We denote the double nested graph by $D(p_1, p_2,\ldots, p_h; q_1, q_2,\ldots, q_h)$
(see Fig. \ref{f1}).
\begin{figure}[H]
\centering
\includegraphics[width=0.5\textwidth]{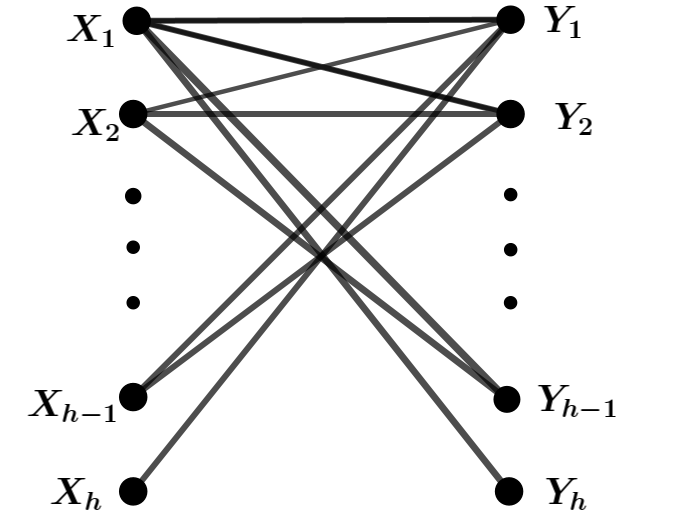}
\caption{Graph $D(p_1, p_2,\ldots, p_h; q_1, q_2,\ldots, q_h).$}\label{f1}
\end{figure}
Define $f(n,r)=\big\lfloor\frac{n-1}{r+1}\big\rfloor\left(n-\big\lfloor\frac{n-1}{r+1}\big\rfloor\right)$ and $g(n,r)=\big\lfloor\frac{n-r-1}{2}\big\rfloor\big\lceil\frac{n-r-1}{2}\big\rceil+r+1.$
Focusing on Problem \ref{pro2}, we prove the following results.
\begin{thm}\label{thm6}
Let $r\geq1$ and $n\geq r^2+r+1$ be two integers, and let $G\in \mathcal{B}_{n,r}.$ Then\\
(i) If $r=1$, then $e(G)\leq f(n,1)$
with equality if and only if $G\cong
K_{n-\big\lfloor\frac{n-1}{2}\big\rfloor, \big\lfloor\frac{n-1}{2}\big\rfloor}.$\\
(ii) If $r\geq2$ and $f(n,r)>g(n,r)$, then $e(G)\leq f(n,r)$ with equality if and
only if $G\cong K_{n-\big\lfloor\frac{n-1}{r+1}\big\rfloor,
\big\lfloor\frac{n-1}{r+1}\big\rfloor}.$\\
(iii) If $r\geq2$ and $f(n,r)<g(n,r)$, then $e(G)\leq g(n,r)$ with equality if and
only if $G\in\left\{ D\left(\big\lceil\frac{n-r-1}{2}\big\rceil,r+1;1,
\big\lfloor\frac{n-r-1}{2}\big\rfloor-1\right),
D\left(\big\lfloor\frac{n-r-1}{2}\big\rfloor,r+1;1,\big\lceil\frac{n-r-1}{2}\big\rceil-1\right)\right\}.$\\
(iv) If $r\geq2$ and $f(n,r)=g(n,r),$ then $e(G)\leq f(n,r)$ with equality if and
only if $G\in\left\{ K_{n-\big\lfloor\frac{n-1}{r+1}\big\rfloor,
\big\lfloor\frac{n-1}{r+1}\big\rfloor},
D\left(\big\lceil\frac{n-r-1}{2}\big\rceil,r+1;1,\big\lfloor\frac{n-r-1}{2}\big\rfloor-1\right), D\left(\big\lfloor\frac{n-r-1}{2}\big\rfloor,r+1;1,\big\lceil\frac{n-r-1}{2}\big\rceil-1\right)\right\}.$
\end{thm}
Let $\rho'=\rho\left(K_{n-\big\lfloor\frac{n-1}{r+1}\big\rfloor,
\big\lfloor\frac{n-1}{r+1}\big\rfloor}\right)$ and
$\rho''=\rho\left(D\left(\big\lceil\frac{n-r-1}{2}\big\rceil,r+1;1,
\big\lfloor\frac{n-r-1}{2}\big\rfloor-1\right)\right).$
\begin{thm}\label{thm7}
Let $r\geq1$ and $n\geq r^2+r+2$ be two integers, and let $G\in \mathcal{B}_{n,r}.$ Then\\
(i) If $r=1$ then
$\rho(G)\leq\rho'$
with equality if and
only if $G\cong
K_{n-\big\lfloor\frac{n-1}{2}\big\rfloor,
\big\lfloor\frac{n-1}{2}\big\rfloor}.$\\
(ii) If $r\geq2$ and $\rho'>\rho'',$ then $\rho(G)\leq \rho'$ with equality if and
only if $G\cong K_{n-\big\lfloor\frac{n-1}{r+1}\big\rfloor,
\big\lfloor\frac{n-1}{r+1}\big\rfloor}.$\\
(iii) If $r\geq2$ and $\rho'<\rho'',$ then $\rho(G)\leq \rho''$ with equality if and
only if $G\cong
D\left(\big\lceil\frac{n-r-1}{2}\big\rceil,r+1;1,\big\lfloor\frac{n-r-1}{2}\big\rfloor-1\right).$\\
(iv) If $r\geq2$ and $\rho'=\rho'',$ then $\rho(G)\leq \rho'$ with equality if and
only if $G\in\left\{ K_{n-\big\lfloor\frac{n-1}{r+1}\big\rfloor,
\big\lfloor\frac{n-1}{r+1}\big\rfloor},
D\left(\big\lceil\frac{n-r-1}{2}\big\rceil,r+1;\big\lfloor\frac{n-r-1}{2}\big\rfloor-1\right)\right\}.$
\end{thm}

\section{Proofs of Theorems \ref{thm4} and \ref{thm5}}
Before proving our main theorems, we present several auxiliary lemmas.

\begin{lem}[Brouwer and Haemers \cite{Brouwer2011}, Godsil and Royle
\cite{Godsil2001}, Haemers \cite{Haemers1995}]\label{le1}
Let $M$ be a real symmetric matrix and $R(M)$ its quotient matrix. The eigenvalues
of every quotient matrix of $M$ interlace those of $M$. If $R(M)$ is equitable,
then the eigenvalues of $R(M)$ are
exactly the eigenvalues of $M$. Furthermore, if $M$ is nonnegative and irreducible,
then the spectral radius of $R(M)$ equals to the spectral radius of $M$.
\end{lem}

\begin{lem}[Godsil and Royle \cite{Godsil2001}] \label{le2}
If $H$ is a spanning subgraph of a connected graph $G,$ then
$\rho(H)\leq\rho(G)$
with equality if and only if $H\cong G.$
\end{lem}

\begin{lem}\label{le3}
$b(K_1\vee(K_{n-r-2}\cup(r+1)K_1))<\frac{1}{r}$.
\end{lem}
\begin{proof}
Let $G=K_1\vee(K_{n-r-2}\cup(r+1)K_1).$ The vertex set of $G$ can be partitioned as $V(G)=V(K_1)\cup V(K_{n-r-2})\cup V((r+1)K_1).$ Take $S=V((r+1)K_1).$  Then $b(G)\leq \frac{|N_G(S)|}{|S|}=\frac{1}{r+1}<\frac{1}{r}.$
\end{proof}

Now we are ready to provide the proof of Theorem \ref{thm4}.

\medskip
\noindent \textbf{Proof of Theorem \ref{thm4}.}
Let $G^*_e$ be the graph with the maximum size in $\mathcal{G}_{n,r}.$ Note that $b(G^*_e)<\frac{1}{r}$ for integer $r\geq1.$ Let $S$ be a binding set. Then $b(G^*_e)=\frac{|N_{G^*_e}(S)|}{|S|}<\frac{1}{r}$ and $N_{G^*_e}(S)\neq V(G^*_e).$ By Theorem \ref{thm1}, $S$ is an independent set, which implies that
$S\cap N_{G^*_e}(S)=\emptyset.$ Define $T_1=N_{G^*_e}(S)$ and $T_2=V(G^*_e)\setminus (S\cup N_{G^*_e}(S)).$ Let $|T_i| =t_i$ for $i=1,2$ and $|S|=s.$ Then we have
\begin{eqnarray}\label{eq1}
s\geq rt_1+1.
\end{eqnarray}
By Lemma \ref{le3}, $K_1\vee(K_{n-r-2}\cup(r+1)K_1)\in \mathcal{G}_{n,r},$ and hence
\begin{eqnarray}\label{equation2}
e(G^*_e)\geq e(K_1\vee(K_{n-r-2}\cup(r+1)K_1))={n-r-1\choose 2}+r+1.
\end{eqnarray}

\begin{claim}\label{claim1}
$G^*_e[T_1\cup T_2]\cong K_{t_1+t_2}$ and $G^*_e[T_1\cup S] \cong K_{t_1}\vee sK_1.$
\end{claim}
\begin{proof}
If $|T_1\cup T_2|=1,$ then the result follows. Next we consider $|T_1\cup T_2|\geq 2.$ Suppose to the contrary that  $G^*_e[T_1\cup T_2]\ncong K_{t_1+t_2}.$ Then there exist two vertices $u, v\in T_1\cup T_2$ such that $uv\notin E(G^*_e).$ Let $G'= G^*_e+uv.$ Then $b(G')\leq \frac{|N_{G'}(S)|}{|S|}=\frac{|N_{G^*_e}(S)|}{|S|}<\frac{1}{r},$ which implies that $G'\in \mathcal{G}_{n,r}.$ However, $e(G')>e(G^*_e),$ which contradicts the maximality of $G^*_e.$ So $G^*_e[T_1\cup T_2]\cong K_{t_1+t_2}.$ Similarly, $G^*_e[T_1\cup S] \cong K_{t_1}\vee sK_1.$
\end{proof}
By Claim \ref{claim1}, we have $G^*_e \cong K_{t_1}\vee(K_{n-t_1-s}\cup sK_1).$

\begin{claim}\label{claim2}
$t_2\geq2.$
\end{claim}
\begin{proof}
By contradiction, assume that $t_2\leq1.$ If $t_2=0,$ then $n=s+t_1.$ By (\ref{eq1}), we have $s\geq\big\lceil\frac{nr+1}{r+1}\big\rceil\geq\frac{nr+1}{r+1}.$
Note that $G^*_e\cong K_{n-s}\vee sK_1.$ Then we have
\begin{eqnarray*}\label{eq2}
e(G^*_e)&=&e(K_{n-s}\vee sK_1)
=\frac{n(n-1)}{2}-\frac{s(s-1)}{2}\\
&\leq&\frac{n(n-1)}{2}-\frac{nr+1}{2(r+1)}\left(\frac{nr+1}{r+1}-1\right)\\
&<&{n-r-1\choose 2}+r+1,
\end{eqnarray*}
which contradicts (\ref{equation2}). Next we consider $t_2=1.$ Let $S'=S\cup T_2.$ Note that $G^*_e$ is connected. Then $N_{G^*_e}(S')=
N_{G^*_e}(S)=T_1,$ and hence $b(G^*_e)\leq\frac{|N_{G^*_e}(S')|}{|S'|}=\frac{t_1}{s+1}<\frac{t_1}{s}=b(G^*_e),$
a contradiction. Hence $t_2\geq2.$
\end{proof}

\begin{claim}\label{claim3}
$s=rt_1+1.$
\end{claim}

\begin{proof}
Note that $b(K_{t_1}\vee(K_{n-(r+1)t_1-1}\cup(rt_1+1)K_1))\leq \frac{t_1}{rt_1+1}<\frac{1}{r}.$ Then $K_{t_1}\vee(K_{n-(r+1)t_1-1}\cup(rt_1+1)K_1)\in \mathcal{G}_{n,r}.$ By (\ref{eq1}), $s\geq rt_1+1.$ If $s\geq rt_1+2,$ then $G^*_e$ is a proper subgraph of $K_{t_1}\vee(K_{n-(r+1)t_1-1}\cup(rt_1+1)K_1).$ Hence $e(K_{t_1}\vee(K_{n-(r+1)t_1-1}\cup(rt_1+1)K_1))>e(G^*_e),$ which contradicts the maximality of $G^*_e.$
\end{proof}

By Claim \ref{claim3}, we have $G^*_e \cong K_{t_1}\vee(K_{n-(r+1)t_1-1}\cup(rt_1+1)K_1).$ Note that $T_1\neq\emptyset.$ Then $t_1\geq1.$ Suppose that $t_1\geq2$. By (\ref{eq1}), $s+t_1+t_2=n,$ and Claim \ref{claim2}, we have $2\leq t_1\leq
\big\lfloor\frac{n-3}{r+1}\big\rfloor.$
Define $$f_1(x)=-\left(\frac{r^2}{2}+r\right)x^2+\left(nr-\frac{3r}{2}-1\right)x+n-1.$$
The quadratic function $f_1(x)$ has axis of symmetry $x=\frac{1}{r^2+2r}
(nr-\frac{3r}{2}-1)>\frac{1}{2}\left(1+\big\lfloor\frac{n-3}{r+1}\big\rfloor\right),$ hence $f_1(x)> f_1(1)=\frac{(r+1)
(2n-r-4)}{2}$ for $2\leq x\leq
\big\lfloor\frac{n-3}{r+1}\big\rfloor$. Combining (\ref{equation2}), we have
\begin{eqnarray*}
{n-r-1\choose 2}+r+1&\leq& e(G^*_e)=e(K_{t_1}\vee(K_{n-(r+1)t_1-1}\cup (rt_1+1)K_1))\\
&=&\frac{n(n-1)}{2}-\left[-\bigg(\frac{r^2}{2}+r\bigg)t_1^2+\bigg(nr-\frac{3r}{2}-1\bigg)t_1+n-1\right]\\
&<&\frac{n(n-1)}{2}-\frac{(r+1) (2n-r-4)}{2}\\
&=&{n-r-1\choose 2}+r+1,
\end{eqnarray*}
a contradiction. Hence $t_1=1.$ Then we have $G^*_e \cong K_1\vee(K_{n-r-2}\cup(r+1)K_1).$
\hspace*{\fill}$\Box$

Now we are ready to present the proof of Theorem \ref{thm5}.

\medskip
\noindent \textbf{Proof of Theorem \ref{thm5}.}
Let $G^*_{\rho}$ be the graph with the maximum spectral radius in $\mathcal{G}_{n,r}.$ Note that $b(G^*_{\rho})<\frac{1}{r}$ for integer $r\geq1.$  Let $S$ be a binding set. Then $b(G^*_{\rho})=\frac{|N_{G^*_{\rho}}(S)|}{|S|}<\frac{1}{r}$ and $N_{G^*_{\rho}}(S)\neq V(G^*_{\rho}).$ By Theorem \ref{thm1}, $S$ is an independent set, which implies that $S\cap N_{G^*_{\rho}}(S)=\emptyset.$ Define $T_1=N_{G^*_{\rho}}(S)$ and $T_2=V(G^*_{\rho})\setminus (S\cup N_{G^*_{\rho}}(S)).$
Let $|T_i| =t_i$ for $i=1,2$ and $|S| =s.$ Then
\begin{eqnarray}\label{eq3}
s\geq rt_1+1.
\end{eqnarray}
By Lemma \ref{le3}, $K_1\vee(K_{n-r-2}\cup(r+1)K_1)\in \mathcal{G}_{n,r},$ and hence
\begin{eqnarray}\label{eq4}
\rho (G^*_{\rho})\geq \rho(K_1\vee(K_{n-r-2}\cup(r+1)K_1))>n-r-2.
\end{eqnarray}

\begin{claim}\label{claim4}
$G^*_{\rho}$ is connected.
\end{claim}

\begin{proof}
Suppose to the contrary that $G^*_{\rho}$ is disconnected. Let $G_1,$ $G_2,\ldots,$ $G_c$ be the connected components of $G^*_{\rho},$ where $c\geq 2.$ Then there exists a connected component $G_i$ such that $S\cap V(G_i)\neq \emptyset.$ Without loss of generality, assume that $S\cap V(G_1)=S_1\neq \emptyset.$ Note that $G^*_{\rho}$ contains no isolated vertices. Then $N_{G_1}(S_1)\neq \emptyset.$  Construct $G'= G^*_{\rho}+uv_2+uv_3+\cdots+uv_c,$ where $u\in V(N_{G_1}(S_1))$ and $v_i\in V(G_i)$ for $2\leq i\leq c.$ Recall that $S\cap N_{G^*_{\rho}}(S)=\emptyset.$ Then $S_1\cap N_{G_1}(S_1)=\emptyset.$ Therefore, $b(G')\leq \frac{|N_{G'}(S)|}{|S|}=\frac{|N_{G^*_{\rho}}(S)|}{|S|}<\frac{1}{r},$ implying that $G'\in \mathcal{G}_{n,r}.$ Note that $G'$ is connected and $G^*_{\rho}$ is a spanning subgraph of $G'.$ By Lemma \ref{le2}, $\rho(G') >\rho(G^*_{\rho}),$ which contradicts the maximality of $G^*_{\rho}.$
\end{proof}

\begin{claim}\label{claim5}
$G^*_{\rho}[T_1\cup T_2]\cong K_{t_1+t_2}$ and $G^*_{\rho}[T_1\cup S] \cong K_{t_1}\vee sK_1.$
\end{claim}
\begin{proof}
If $|T_1\cup T_2|=1,$ then the result follows. Next we consider $|T_1\cup T_2|\geq 2.$ Suppose to the contrary that  $G^*_{\rho}[T_1\cup T_2]\ncong K_{t_1+t_2}.$ Then there exist two vertices $u, v\in T_1\cup T_2$ such that $uv\notin E(G^*_{\rho}).$ Let $G''= G^*_{\rho}+uv.$ Then $b(G'')\leq \frac{|N_{G''}(S)|}{|S|}=\frac{|N_{G^*_{\rho}}(S)|}{|S|}<\frac{1}{r},$ which implies that $G''\in \mathcal{G}_{n,r}.$ However, by Claim \ref{claim4} and Lemma \ref{le2}, $\rho(G'')>\rho(G^*_{\rho}),$  which contradicts the maximality of $G^*_{\rho}.$ So $G^*_{\rho}[T_1\cup T_2]\cong K_{t_1+t_2}.$ Similarly, $G^*_{\rho}[T_1\cup S] \cong K_{t_1}\vee sK_1.$
\end{proof}
By Claim \ref{claim5}, we have $G^*_{\rho} \cong K_{t_1}\vee(K_{n-t_1-s}\cup sK_1).$

\begin{claim}\label{claim6}
$t_2\geq2.$
\end{claim}

\begin{proof}
By contradiction, assume that $t_2\leq1.$ If $t_2=0,$ then $n=s+t_1.$ By (\ref{eq3}), we deduce that $s\geq\left\lceil\frac{nr+1}{r+1}\right\rceil=n-\big\lfloor\frac{n-1}{r+1}\big\rfloor.$ Note that $G^*_{\rho}\cong K_{n-s}\vee sK_1.$ For convenience, let $\alpha=\big\lfloor\frac{n-1}{r+1}\big\rfloor.$ Note that $K_{n-s}\vee sK_1$ is a spanning subgraph of $K_{\alpha}\vee(n-\alpha)K_1.$
By (\ref{eq4}) and Lemma \ref{le2}, we have
\begin{eqnarray}\label{eq5}
n-r-2<\rho(G^*_{\rho})=\rho(K_{n-s}\vee s K_1)\leq\rho(K_{\alpha}\vee(n-\alpha)K_1).
\end{eqnarray}
Notice that $A(K_{\alpha}\vee(n-\alpha)K_1)$ has an equitable quotient matrix
$$R(A(K_{\alpha}\vee(n-\alpha)K_1))=\left[
\begin{array}{cc}
\alpha-1&n-\alpha\\
\alpha&0
\end{array}
\right].
$$
Then the characteristic polynomial of $R(A(K_{\alpha}\vee(n-\alpha)K_1))$ is $f_2(x)=x^2-(\alpha-1)x+\alpha(\alpha-n).$
The axis of symmetry of $f_2(x)$ is $x=\frac{\alpha-1}{2}<n-r-2,$ which implies that $f_2(x)$ is increasing for $x\geq n-r-2.$
By $n\geq2r+15,$ we have
\begin{eqnarray*}
f_2(x)\geq f_2(n-r-2)=n(n-2r-2\alpha-3)+r^2+(\alpha+3)r+{\alpha}^2+2\alpha+2>0.
\end{eqnarray*}
Hence $\rho(K_{\alpha}\vee(n-\alpha)K_1)<n-r-2.$ Combining (\ref{eq5}), we have
\begin{eqnarray*}
n-r-2<\rho(G^*_{\rho})\leq\rho(K_{\alpha}\vee(n-\alpha)K_1)<n-r-2,
\end{eqnarray*}
a contradiction. Next we consider $t_2=1.$ Let $S^*=S\cup T_2.$ By Claim \ref{claim4},  $G^*_{\rho}$ is connected. Then $N_{G^*_{\rho}}(S^*)=N_{G^*_{\rho}}(S)=T_1,$ and hence $b(G^*_{\rho})\leq\frac{|N_{G^*_{\rho}}(S^*)|}{|S^*|}=\frac{t_1}{s+1}<\frac{t_1}{s}=b(G^*_{\rho}),$ a contradiction. Hence $t_2\geq2.$
\end{proof}

\begin{claim}\label{claim7}
$s=rt_1+1.$
\end{claim}

\begin{proof}
Note that $b(K_{t_1}\vee(K_{n-(r+1)t_1-1}\cup(rt_1+1)K_1))\leq \frac{t_1}{rt_1+1}<\frac{1}{r}.$ Then $K_{t_1}\vee(K_{n-(r+1)t_1-1}\cup(rt_1+1)K_1)\in \mathcal{G}_{n,r}.$ By (\ref{eq3}), $s\geq rt_1+1.$ If $s\geq rt_1+2,$ then $G^*_{\rho}$ is a spanning subgraph of $K_{t_1}\vee(K_{n-(r+1)t_1-1}\cup(rt_1+1)K_1).$ By Lemma \ref{le2}, $\rho(K_{t_1}\vee(K_{n-(r+1)t_1-1}\cup(rt_1+1)K_1))>\rho(G^*_{\rho}).$ This contradicts the maximality of $G^*_{\rho}.$ Hence $s=rt_1+1.$
\end{proof}
By Claim \ref{claim7}, $G^*_{\rho}\cong K_{t_1}\vee(K_{n-(r+1)t_1-1}\cup(rt_1+1)K_1).$ Combining $s+t_1+t_2=n,$ Claim \ref{claim6} and Claim \ref{claim7}, we deduce that $1\leq t_1\leq\big\lfloor\frac{n-3}{r+1}\big\rfloor.$ Next our goal is to prove that $t_1=1.$ Suppose to the contrary that $2\leq t_1\leq\big\lfloor\frac{n-3}{r+1}\big\rfloor.$ Note that $A(G^*_{\rho})$ has an equitable quotient matrix
$$
R(A(G^*_{\rho}))=\left[
\begin{array}{ccc}
t_1-1&n-(r+1)t_1-1&rt_1+1\\
t_1&n-(r+1)t_1-2&0\\
t_1&0&0
\end{array}
\right].
$$
Then the characteristic polynomial of $R(A(G^*_{\rho}))$ is $$f_3(x)=x^3-(n-rt_1-3)x^2-(rt_1^2-(r-1)t_1+n-2)x-(r^2+r)t_1^3+((n-3)r-1)t_1^2+(n-2)t_1.$$
Let $G'''=K_1\vee(K_{n-r-2}\cup(r+1)K_1).$ Note that $A(G''')$ has an equitable quotient matrix
$$
R(A(G'''))=\left[
\begin{array}{ccc}
0&n-r-2&r+1\\
1&n-r-3&0\\
1&0&0
\end{array}
\right].
$$
Its characteristic polynomial is
$$f_4(x)=x^3-(n-r-3)x^2-(n-1)x+(r+1)(n-r-3).$$ By (\ref{eq4}), $\rho(K_1\vee(K_{n-r-2}\cup(r+1)K_1))>n-r-2.$ Next we prove that $f_3(x)-f_4(x)>0$ for $x\geq n-r-2.$ Let $f_5(x)=rx^2-(rt_1+1)x-(t_1^2+t_1+1)r^2+(nt_1-t_1^2+n-4t_1-4)r+n-t_1-3.$ Then $f_3(x)-f_4(x)=(t_1-1)f_5(x).$ Since $n\geq2r+15,$ the axis of symmetry of $f_5(x)$ is $x=\frac{rt_1+1}{2r}<n-r-2,$ which implies that $f_5(x)$ is increasing for $x\geq
n-r-2$. Then $f_5(x)\geq f_5(n-r-2)=-((r^2+r)t_1^2+(2r+1)t_1-r^3+(2n-3)r^2-(n^2-3n+1)r+1)\geq\frac{n^2r^2-(2r^3+5r^2-r+1)n+r^4+4r^3+4r^2-3r+2}{r+1}>0$ for $x\geq n-r-2.$ Hence $f_3(x)-f_4(x)>0$ for $x\geq n-r-2.$ By Lemma \ref{le1}, we have
\begin{eqnarray*}
\rho(K_1\vee(K_{n-r-2}\cup(r+1)K_1))=\lambda_1(R(A(G'''))>\lambda_1(R(A(G^*_{\rho}))=\rho(G^*_{\rho}),
\end{eqnarray*}
which contradicts (\ref{eq4}). Hence $t_1=1,$ and so $G^*_{\rho}\cong K_1\vee(K_{n-r-2}\cup(r+1)K_1).$
\hspace*{\fill}$\Box$

\section{Proof of Theorem \ref{thm6}}
Before proceeding with the proof, we first state several useful lemmas.

\begin{lem}\label{le4}
Let $G=(X,Y)$ be a bipartite graph, and let $S$ be a binding set of $G.$ If $S\cap X\neq \emptyset$ and $S\cap Y\neq \emptyset,$ then $S\cap X$ and $S\cap Y$ are also binding sets.
\end{lem}
\begin{proof}
Let $S_1=S\cap X$ and  $S_2=S\cap Y.$ Since $S$ is a binding set of $G,$ $b(G)=\frac{|N_G(S)|}{|S|}=\frac{|N_G(S_1)|+|N_G(S_2)|}{|S_1|+|S_2|}.$ For short, let $b(G)=b.$ By the definition of binding number, we have $\frac{|N_G(S_1)|}{|S_1|}\geq b$ and $ \frac{|N_G(S_2)|}{|S_2|}\geq b.$ We claim that $\frac{|N_G(S_1)|}{|S_1|}=\frac{|N_G(S_2)|}{|S_2|}=b.$ In fact, if $\frac{|N_G(S_1)|}{|S_1|}>b$ or $ \frac{|N_G(S_2)|}{|S_2|}> b,$ then $|N_G(S_1)|+|N_G(S_2)|> b(|S_1|+|S_2|),$ and hence $$b=b(G)=\frac{|N_G(S_1)|+|N_G(S_2)|}{|S_1|+|S_2|}>b,$$ a contradiction. This implies that $S\cap X$ and $S\cap Y$ are binding sets.
\end{proof}

\begin{lem}\label{le5}
Let $p,q,r$ be three positive integers. Then $b(D(p-r(q-1)-1,r(q-1)+1; q-1,1))< \frac{1}{r},$ $b(D(p-r-1,r+1;1,q-1)<\frac{1}{r}$ and $b(K_{n-\big\lfloor\frac{n-1}{r+1}\big\rfloor, \big\lfloor\frac{n-1}{r+1}\big\rfloor})<\frac{1}{r}.$
\end{lem}
\begin{proof}
Let $G_1=D(p-r(q-1)-1,r(q-1)+1; q-1,1),$ $G_2=D(p-r-1,r+1;1,q-1)$ and $G_3=K_{n-\big\lfloor\frac{n-1}{r+1}\big\rfloor,
\big\lfloor\frac{n-1}{r+1}\big\rfloor}.$ The vertex set of $G_1$ can be partitioned as $V(G_1)=V((p-r(q-1)-1)K_1) \cup V((r(q-1)+1)K_1) \cup  V((q-1)K_1)\cup V(K_1).$ Take $S_1=V(((r(q-1)+1)K_1).$ Then $b(G_1)\leq \frac{|N_{G_1}(S_1)|}{|S_1|}=\frac{q-1}{r(q-1)+1}<\frac{1}{r}.$ The vertex set of $G_2$ can be partitioned as $V(G_2)=V((p-r-1)K_1)\cup V((r+1)K_1)\cup V(K_1)\cup V((q-1)K_1).$ Take $S_2=V((r+1)K_1).$ Then $b(G_2)\leq \frac{|N_{G_2}(S_2)|}{|S_2|}=\frac{1}{r+1}<\frac{1}{r}.$ The vertex set of $G_3$ can be partitioned as $V(G_3)=V((n-\big\lfloor\frac{n-1}{r+1}\big\rfloor)K_1)\cup V((\big\lfloor\frac{n-1}{r+1}\big\rfloor)K_1).$ Take $S_3=V((n-\big\lfloor\frac{n-1}{r+1}\big\rfloor)K_1).$ Then $b(G_3)\leq \frac{|N_{G_3}(S_3)|}{|S_3|}=\frac{\big\lfloor\frac{n-1}{r+1}\big\rfloor}{n-\big\lfloor\frac{n-1}{r+1}\big\rfloor}\leq \frac{n-1}{rn+1}<\frac{1}{r}.$
\end{proof}

Next we prove a technical lemma which is very important to our main result. Let $\mathcal{B}_{p,q,r}$ be the family of $n$-vertex bipartite graphs with the bipartition $(X,Y)$ that contain no isolated vertices and satisfy $b(G)<\frac{1}{r}$ for each $G\in\mathcal{B}_{p,q,r}$, where $r\geq 1$ is an integer. Without loss of generality, assume that $|X|=p\geq|Y|=q.$
\begin{lem}\label{le6}
Let $r\geq1$ be an integer, and let $G\in \mathcal{B}_{p,q,r}.$ Each of the following holds.\\
(i) If $p\geq rq+1,$ then $e(G)\leq pq$ with equality if and only if $G\cong K_{p,q}.$\\
(ii) If $r(q-1)+2\leq p\leq rq,$ then $e(G)\leq pq-r(q-1)-1$ with equality if and only if $G\cong D(p-r(q-1)-1,r(q-1)+1; q-1,1).$\\
(iii) If $p\leq r(q-1)+1,$ then $e(G)\leq pq-(r+1)(q-1)$ with equality if and only if $G\cong D(p-r-1,r+1;1,q-1).$
\end{lem}

\begin{proof}
Let $B^*_e$ be the graph with the maximum size in $\mathcal{B}_{p,q,r}.$ Note that $b(B^*_e)<\frac{1}{r}$ for integer $r\geq1.$ By Lemma \ref{le4}, we can choose a binding set $S$ such that $S\subseteq X$ or $S\subseteq Y.$ Then $b(B^*_e)=\frac{|N_{B^*_e}(S)|}{|S|}<\frac{1}{r}$ and $N_{B^*_e}(S)\neq V(B^*_e).$ Let $|S|=s$ and $|N_{B^*_e}(S)|=t.$ Then we have
\begin{eqnarray}\label{eq6}
s\geq rt+1.
\end{eqnarray}

\begin{claim}\label{claim8}
If $S\subseteq X,$ then $B^*_e[(X/S)\cup Y]\cong K_{p-s,q}$ and $B^*_e[S\cup N_{B^*_e}(S)] \cong K_{s,t}.$
\end{claim}
\begin{proof}
Suppose to the contrary that $B^*_e[(X/S)\cup Y]\ncong K_{p-s,q}.$ Then there exist two vertices $u, v$ such that $uv\notin E(B^*_e),$ where $u\in X/S$ and $v\in Y.$ Let $B'= B^*_e+uv.$ Then $b(B')\leq \frac{|N_{B'}(S)|}{|S|}=\frac{|N_{B^*_e}(S)|}{|S|}<\frac{1}{r},$ which implies that $B'\in \mathcal{B}_{p,q,r}.$ However, $e(B')>e(B^*_e),$ which contradicts the maximality of $B^*_e.$ So $B^*_e[(X/S)\cup Y]\cong K_{p-s,q}.$ Similarly, we can prove that $B^*_e[S\cup N_{B^*_e}(S)]\cong K_{s,t}.$
\end{proof}
Similar to the analysis of Claim \ref{claim8}, one can immediately obtain the following result.
\begin{claim}\label{claim8.0}
If $S\subseteq Y,$ then $B^*_e[X\cup (Y/S)]\cong K_{p,q-s}$ and $B^*_e[S\cup N_{B^*_e}(S)] \cong K_{t,s}.$
\end{claim}

By Claim \ref{claim8} and Claim \ref{claim8.0}, $B_e^*\cong D(p-s,s;t,q-t)$ for $S\subseteq X$ and $B_e^*\cong D(t,p-t;q-s,s)$ for $S\subseteq Y$. Hence $B_e^*$ is connected.

\begin{claim}\label{claim9}
$s=rt+1.$
\end{claim}

\begin{proof}
We first consider $S\subseteq X.$ Note that $b(D(p-rt-1,rt+1;t,q-t))\leq \frac{t}{rt+1}<\frac{1}{r}.$ Then $D(p-rt-1,rt+1;t,q-t)\in \mathcal{B}_{p,q,r}.$ By (\ref{eq6}), $s\geq rt+1.$ If $s\geq rt+2,$ then $B^*_e$ is a proper subgraph of $D(p-rt-1,rt+1;t,q-t).$ It follows that $e(D(p-rt-1,rt+1;t,q-t))>e(B^*_e),$ which contradicts the maximality of $B^*_e.$ Therefore, $s=rt+1.$ By the same analysis, the result holds for $S \subseteq Y.$
\end{proof}
By Claim \ref{claim9}, $B^*_e\cong D(p-rt-1,rt+1;t,q-t)$ for $S\subseteq X$ and $B^*_e\cong D(t,p-t;q-rt-1,rt+1)$ for $S\subseteq Y.$

\begin{claim}\label{claim10}
If $S\subseteq Y$ and $p\leq rq,$ then $t=1.$
\end{claim}
\begin{proof}
Note that $S\subseteq Y.$ Then $s\leq q.$ By Claim \ref{claim9}, we have $t<s\leq q\leq p,$ and hence $t\leq p-1.$ If $s=q,$ then $B^*_e$ is disconnected, a contradiction. So $s\leq q-1.$ Recall that $s=rt+1.$ Then $1\leq t\leq\big\lfloor\frac{q-2}{r}\big\rfloor.$ Note that $b(D(1,p-1;q-r-1,r+1))\leq \frac{1}{r+1}<\frac{1}{r}.$ Then $D(1,p-1;q-r-1,r+1)\in \mathcal{B}_{p,q,r},$ and hence $e(B^*_e)\geq e(D(1,p-1;q-r-1,r+1))=pq-(r+1)(p-1).$ If $1< t\leq\big\lfloor\frac{q-2}{r}\big\rfloor,$ then
\begin{eqnarray*}
pq-(r+1)(p-1)\leq e(B^*_e)
&=&e(D(t, p-t; q-rt-1, rt+1))\\
&=&pq+rt^2+(1-rp)t-p\\
&<&pq-(r+1)(p-1),
\end{eqnarray*}
a contradiction. Hence $t=1.$
\end{proof}

\begin{claim}\label{claim11}
If $S\subseteq X$ and $p\leq rq,$ then $1\leq t \leq \big\lfloor\frac{p-2}{r}\big\rfloor.$
\end{claim}
\begin{proof}
Since $S\subseteq X,$ we have $N_{B^*_e}(S)\subseteq Y,$ and hence $t\leq q.$ By $p\leq rq,$ Claim \ref{claim9} and $s\leq p,$ we have  $1\leq t\leq\big\lfloor\frac{p-1}{r}\big\rfloor.$ If $q=t<s=p,$ then $b(B^*_e)=\frac{|N_{B^*_e}(S)|}{|S|}=\frac{t}{s}=\frac{q}{p}\geq\frac{1}{r},$ a contradiction. If $t\leq q-1<s=p,$ then $B^*_e$ is disconnected, a contradiction. Hence $s\leq p-1.$ Since $s=rt+1,$ we have $1\leq t \leq \big\lfloor\frac{p-2}{r}\big\rfloor.$
\end{proof}

(i) Since $p\geq rq+1,$ $b(K_{p,q})=\frac{q}{p}\leq\frac{q}{rq+1}<\frac{1}{r},$ and hence $K_{p,q}\in \mathcal{B}_{p,q,r}.$ So we have $e(B^*_e)\geq e(K_{p,q})=pq.$ Note that $B^*_e$ is a subgraph of $K_{p,q}.$ Then $B^*_e\cong K_{p,q}.$

(ii) By Lemma \ref{le5}, $D(p-r(q-1)-1,r(q-1)+1; q-1,1)\in \mathcal{B}_{p,q,r}.$ Therefore,
\begin{eqnarray}\label{eq7}
e(B^*_e)\geq e(D(p-r(q-1)-1,r(q-1)+1; q-1,1))=pq-r(q-1)-1.
\end{eqnarray}
Next we will prove that $S\subseteq X.$ If $S\subseteq Y,$  by Claim \ref{claim10} and $B^*_e\cong D(t,p-t;q-rt-1,rt+1),$ we have $$e(B^*_e)= pq-(r+1)(p-1)< pq-r(q-1)-1,$$ which contradicts (\ref{eq7}). Hence $S\subseteq X.$ Since $r(q-1)+2\leq p\leq rq,$ we have $r\geq2$ and $q-1\leq \big\lfloor\frac{p-2}{r}\big\rfloor.$ Moreover, as $r\geq2$ and $p\leq rq,$ we obtain that $q-1\geq \big\lfloor\frac{p-2}{r}\big\rfloor.$ Hence $q-1=\big\lfloor\frac{p-2}{r}\big\rfloor.$ By Claim \ref{claim11}, we have $1\leq t\leq q-1.$ Next we claim that $t=q-1.$ In fact, if $t<q-1,$ then
\begin{eqnarray*}
e(B^*_e)&=&e(D(p-rt-1,rt+1; t,q-t))\\
&=&pq+rt^2+(1-rq)t-q\\
&<&pq-r(q-1)-1,
\end{eqnarray*}
which contradicts (\ref{eq7}). Hence $t=q-1.$ Then we have $B^*_e\cong D(p-r(q-1)-1, r(q-1)+1; q-1,1).$

(iii) By Lemma \ref{le5}, $D(p-r-1,r+1;1,q-1)\in \mathcal{B}_{p,q,r},$ and hence
\begin{eqnarray}\label{eq8}
e(B^*_e)\geq e(D(p-r-1,r+1;1,q-1))=pq-(r+1)(q-1).
\end{eqnarray}
We distinguish our proof into the following two cases.

\vspace{1.5mm}
\noindent\textbf{Case 1.} $S\subseteq X.$
\vspace{1mm}

By $p\leq r(q-1)+1\leq rq$ and Claim \ref{claim11}, we have $1\leq t\leq\big\lfloor\frac{p-2}{r}\big\rfloor\leq q-2.$ Note that $B^*_e\cong D(p-rt-1,rt+1; t,q-t).$ If $1<t\leq q-2,$ then
\begin{eqnarray*}
e(B^*_e)&=&e(D(p-rt-1,rt+1; t,q-t))\\
&=&pq+rt^2+(1-rq)t-q\\
&<&pq-(r+1)(q-1),
\end{eqnarray*}
which contradicts (\ref{eq8}). Hence $t=1.$ Then we have $B^*_e\cong D(p-r-1, r+1; 1,q-1).$

\vspace{1.5mm}
\noindent\textbf{Case 2.} $S\subseteq Y.$
\vspace{1mm}

By $q\leq p\leq r(q-1)+1\leq rq,$ Claim \ref{claim10} and $B^*_e\cong D(t,p-t;q-rt-1,rt+1),$ we have $e(B^*_e)=pq-(r+1)(p-1)\leq pq-(r+1)(q-1).$ If $p>q,$ then $e(B^*_e)=pq-(r+1)(p-1)< pq-(r+1)(q-1),$ which contradicts (\ref{eq8}). Hence $p=q.$ Then $B^*_e\cong D(p-r-1, r+1; 1,q-1).$
\end{proof}

Recall that $f(n,r)=\big\lfloor\frac{n-1}{r+1}\big\rfloor\left(n-\big\lfloor\frac{n-1}{r+1}\big\rfloor\right)$ and $g(n,r)=\big\lfloor\frac{n-r-1}{2}\big\rfloor\big\lceil\frac{n-r-1}{2}\big\rceil+r+1.$

\medskip
\noindent \textbf{Proof of Theorem \ref{thm6}.}
Let $\tilde{B^*_e}=(X,Y)$ be the bipartite graph with the maximum size in $\mathcal{B}_{n,r}.$ Without loss of generality, we assume that $|X|\geq |Y|.$ Let $|Y|=q.$ Then $|X|=n-q$ and $q\leq \big\lfloor\frac{n}{2}\big\rfloor.$

(i) Note that $r=1.$ We claim that $|X|=n-q\geq q+1.$ In fact, if $n-q\leq q,$ then $n=2q,$ which implies that $n$ is even. By Lemma \ref{le6} (iii), we have $\tilde{B^*_e}\cong D(\frac{n}{2}-2,2;1,\frac{n}{2}-1).$ Then $e(\tilde{B^*_e})=e(D(\frac{n}{2}-2,2;1,\frac{n}{2}-1))=\frac{n^2}{4}-n+2.$ According to Lemma \ref{le5},  $K_{n-\big\lfloor\frac{n-1}{2}\big\rfloor,\big\lfloor\frac{n-1}{2}\big\rfloor}\in \mathcal{B}_{n,1}$ and $e(K_{n-\big\lfloor\frac{n-1}{2}\big\rfloor,\big\lfloor\frac{n-1}{2}\big\rfloor})=\left(n-\big\lfloor\frac{n-1}{2}\big\rfloor\right)\big\lfloor\frac{n-1}{2}\big\rfloor.$ Note that $n$ is even. Then $$\left(n-\Big\lfloor\frac{n-1}{2}\Big\rfloor\right)\Big\lfloor\frac{n-1}{2}\Big\rfloor\leq e(\tilde{B^*_e})=\frac{n^2}{4}-n+2<\frac{n^2-4}{4}=\left(n-\Big\lfloor\frac{n-1}{2}\Big\rfloor\right)\Big\lfloor\frac{n-1}{2}\Big\rfloor,$$ a contradiction. Hence $n-q\geq q+1,$ and $q\leq\big\lfloor\frac{n-1}{2}\big\rfloor.$ By Lemma \ref{le6} (i), we have $\tilde{B^*_e}\cong K_{n-q,q}.$ Since $q\leq\big\lfloor\frac{n-1}{2}\big\rfloor,$ $\tilde{B^*_e}\cong K_{n-\big\lfloor\frac{n-1}{2}\big\rfloor,\big\lfloor\frac{n-1}{2}\big\rfloor}.$

(ii) Note that $r\geq2.$ We first prove that $n-q\geq rq+1.$ Suppose to the contrary that $n-q\leq rq.$ If $r(q-1)+2\leq n-q\leq rq,$ then $\big\lceil\frac{n}{r+1}\big\rceil\leq q\leq\big\lfloor\frac{n+r-2}{r+1}\big\rfloor.$ By Lemma \ref{le6} (ii), we have $\tilde{B^*_e}\cong D(n-(r+1)q+r-1,r(q-1)+1; q-1,1).$ Note that $K_{n-\big\lfloor\frac{n-1}{r+1}\big\rfloor, \big\lfloor\frac{n-1}{r+1}\big\rfloor}\in \mathcal{B}_{n,r}$ and $e(K_{n-\big\lfloor\frac{n-1}{r+1}\big\rfloor, \big\lfloor\frac{n-1}{r+1}\big\rfloor})=(n-\big\lfloor\frac{n-1}{r+1}\big\rfloor)\big\lfloor\frac{n-1}{r+1}\big\rfloor.$ Then $e(\tilde{B^*_e})\geq (n-\big\lfloor\frac{n-1}{r+1}\big\rfloor)\big\lfloor\frac{n-1}{r+1}\big\rfloor=f(n,r).$ By $\big\lceil\frac{n}{r+1}\big\rceil\leq q\leq\big\lfloor\frac{n+r-2}{r+1}\big\rfloor$ and $n\geq r^2+r+1,$ we have $$f(n,r)\leq e(\tilde{B^*_e})= (n-q)q-r(q-1)-1\leq\Big\lfloor\frac{n+r-2}{r+1}\Big\rfloor\left(n-r-\Big\lfloor\frac{n+r-2}{r+1}\Big\rfloor\right)+r-1<f(n,r),$$ a contradiction. If $n-q\leq r(q-1)+1,$ then $\big\lfloor\frac{n+r-1}{r+1}\big\rfloor\leq q\leq\lfloor\frac{n}{2}\rfloor.$ By Lemma \ref{le6} (iii), we have $\tilde{B^*_e}\cong D(n-q-r-1,r+1;1,q-1).$ By $f(n,r)>g(n,r)$ and $\big\lfloor\frac{n+r-1}{r+1}\big\rfloor\leq q\leq\lfloor\frac{n}{2}\rfloor,$ we can obtain that
$$f(n,r)\leq e(\tilde{B^*_e})=(n-q)q-(r+1)(q-1)\leq\Big\lfloor\frac{n-r-1}{2}\Big\rfloor\Big\lceil\frac{n-r-1}{2}\Big\rceil+r+1<f(n,r),$$ a contradiction. Hence $n-q\geq rq+1,$ and $1\leq q\leq\big\lfloor\frac{n-1}{r+1}\big\rfloor.$ By Lemma \ref{le6} (i), we have $\tilde{B^*_e}\cong K_{n-q,q}.$ If $q<\big\lfloor\frac{n-1}{r+1}\big\rfloor,$ then $$f(n,r)\leq e(\tilde{B^*_e})=e(K_{n-q,q})=(n-q)q<f(n,r),$$ a contradiction. Therefore, $q=\big\lfloor\frac{n-1}{r+1}\big\rfloor,$ which implies that $\tilde{B^*_e}\cong K_{n-\big\lfloor\frac{n-1}{r+1}\big\rfloor,\big\lfloor\frac{n-1}{r+1}\big\rfloor}.$

(iii) Note that $r\geq2.$ We claim that $n-q\leq r(q-1)+1.$ By contradiction, assume that $n-q\geq r(q-1)+2.$  If $r(q-1)+2\leq n-q\leq rq,$ then $\big\lceil\frac{n}{r+1}\big\rceil\leq q\leq\big\lfloor\frac{n+r-2}{r+1}\big\rfloor.$ Combining Lemma \ref{le6} (ii), we have $\tilde{B^*_e}\cong D(n-(r+1)q+r-1,r(q-1)+1; q-1,1).$ Note that $D\big(\big\lceil\frac{n-r-1}{2}\big\rceil,r+1;1, \big\lfloor\frac{n-r-1}{2}\big\rfloor-1\big)\in \mathcal{B}_{n,r}$ and $e\big(D\big(\big\lceil\frac{n-r-1}{2}\big\rceil,r$ $+1;1, \big\lfloor\frac{n-r-1}{2}\big\rfloor-1\big)\big)=g(n,r).$ Then $e(\tilde{B^*_e})\geq g(n,r).$ By $\big\lceil\frac{n}{r+1}\big\rceil\leq q\leq\big\lfloor\frac{n+r-2}{r+1}\big\rfloor$ and $n\geq r^2+r+1,$ we have $$g(n,r)\leq e(\tilde{B^*_e})= (n-q)q-r(q-1)-1\leq\Big\lfloor\frac{n+r-2}{r+1}\Big\rfloor\left(n-r-\Big\lfloor\frac{n+r-2}{r+1}\Big\rfloor\right)+r-1<g(n,r),$$
a contradiction. If $n-q\geq rq+1,$ then $1\leq q\leq\big\lfloor\frac{n-1}{r+1}\big\rfloor.$ By Lemma \ref{le6} (i), we have $\tilde{B^*_e}\cong K_{n-q,q}.$ By $1\leq q\leq\big\lfloor\frac{n-1}{r+1}\big\rfloor$ and $f(n,r)<g(n,r),$ we have
$$g(n,r)\leq e(\tilde{B^*_e})= (n-q)q\leq\left(n-\Big\lfloor\frac{n-1}{r+1}\Big\rfloor\right)\Big\lfloor\frac{n-1}{r+1}\Big\rfloor=f(n,r)<g(n,r),$$ a contradiction. Hence $n-q\leq r(q-1)+1,$ and $\big\lfloor\frac{n+r-1}{r+1}\big\rfloor\leq q\leq \lfloor\frac{n}{2}\rfloor.$ By Lemma \ref{le6} (iii), we have $\tilde{B^*_e}\cong D(n-q-r-1,r+1;1,q-1).$ If $q\neq\big\lfloor\frac{n-r-1}{2}\big\rfloor$ and $q\neq\big\lceil\frac{n-r-1}{2}\big\rceil,$ then $$g(n,r)\leq e(\tilde{B^*_e})= (n-q)q-(r+1)(q-1)<\Big\lfloor\frac{n-r-1}{2}\Big\rfloor\Big\lceil\frac{n-r-1}{2}\Big\rceil+r+1=g(n,r),$$ a contradiction. Hence we have $q=\big\lfloor\frac{n-r-1}{2}\big\rfloor$ or $q=\big\lceil\frac{n-r-1}{2}\big\rceil,$ which implies that $\tilde{B^*_e}\cong D\Big(\big\lceil\frac{n-r-1}{2}\big\rceil,r+1;1, \big\lfloor\frac{n-r-1}{2}\big\rfloor-1\Big)$ or $\tilde{B^*_e}\cong D\Big(\big\lfloor\frac{n-r-1}{2}\big\rfloor,r+1;1,\big\lceil\frac{n-r-1}{2}\big\rceil-1\Big).$

(iv) Based on the proofs of (ii) and (iii), (iv) directly follows.\hspace*{\fill}$\Box$

\section{Proof of Theorem \ref{thm7}}

Before presenting the proof, we state some necessary lemmas.

\begin{lem}[Bhattacharya et al.\cite{Bhattacharya2008}] \label{le7}
If $G$ is a bipartite graph, then $\rho(G) \leq \sqrt{e(G)}$  with equality holds if and only if $G$ is a complete bipartite graph.
\end{lem}

Recall that $\mathcal{B}_{p,q,r}$ is the family of $n$-vertex bipartite graphs with the bipartition $(X,Y)$ that contain no isolated vertices and satisfy $b(G)<\frac{1}{r}$ for each $G\in\mathcal{B}_{p,q,r}$, where $r\geq 1$ is an integer and $|X|=p\geq|Y|=q.$

\begin{lem}\label{le7.0}
Let $2\leq t\leq\big\lfloor\frac{p-2}{r}\big\rfloor$ be an integer. If $\big\lfloor\frac{p-2}{r}\big\rfloor\leq q-3,$ then $\rho(D(p-rt-1,rt+1; t,q-t))<\rho(D(p-r-1, r+1; 1, q-1)).$
\end{lem}

\begin{proof}
Since $\big\lfloor\frac{p-2}{r}\big\rfloor\leq q-3,$ $p\leq r(q-2)+1.$ Define $G_1=D(p-rt-1,rt+1; t,q-t).$ Then the vertex set of $G_1$ can be partitioned as $V(G_1)=V(X_1)\cup  V(X_2)\cup V(Y_1)\cup V(Y_2),$ where $V(X_1)=\{u_1,u_2,\ldots,u_{p-rt-1}\}$, $V(X_2)=\{v_1,v_2,\ldots,v_{rt+1}\}$, $V(Y_1)=\{w_1,w_2,\ldots,w_t\}$ and $V(Y_2)=\{z_1,z_2,\ldots,z_{q-t}\}$. Let $x$ be the Perron vector of $A(G_1)$, and let $\rho=\rho(G_1)$. By symmetry, $x$ takes the same value (say $x_1$, $x_2$, $x_3$, $x_4$) on the vertices of $V(X_1),$ $V(X_2)$, $V(Y_1)$, $V(Y_2),$ respectively. By $A(G_1)x=\rho x$, we have
\begin{eqnarray*}
&&\rho x_1=tx_3+(q-t)x_4,\\
&&\rho x_2=tx_3,\\
&&\rho x_3=(p-rt-1)x_1+(rt+1)x_2,\\
&&\rho x_4=(p-rt-1)x_1,
\end{eqnarray*}
which leads to
\begin{eqnarray*}
x_2=\frac{t}{\rho}x_3,\ x_4=\left(1-\frac{t(rt+1)}{\rho^2}\right)x_3.
\end{eqnarray*}
Let $G_2=D(p-r-1, r+1; 1, q-1).$ Then the vertex set of $G_2$ can be partitioned as $V(G_2)=V(X_1')\cup  V(X_2')\cup V(Y_1')\cup V(Y_2').$ Let $y$ be the Perron vector of $A(G_2)$, and let $\rho'= \rho(G_2)$. By symmetry, $y$ takes the same value (say $y_1$, $y_2$, $y_3$, $y_4$) on the vertices of $V(X_1'),$ $V(X_2')$, $V(Y_1')$, $V(Y_2'),$ respectively. Then, by
$A(G_2)y=\rho' y$, we have
\begin{eqnarray*}
&&\rho' y_1=y_3+(q-1)y_4,\\
&&\rho' y_2=y_3,\\
&&\rho' y_3=(p-r-1)y_1+(r+1)y_2,\\
&&\rho' y_4=(p-r-1)y_1,
\end{eqnarray*}
which leads to
\begin{eqnarray*}
y_2=\left(1-\frac{(p-r-1) (q-1)}{\rho'^2}\right)y_1,\
y_4=\frac{p-r-1}{\rho'}y_1.
\end{eqnarray*}

Let $E_1=\{ v_iz_j\mid r+2\leq i\leq rt+1,  1\leq j\leq q-t\} $ and $E_2=\{v_iw_j\mid 1\leq i\leq r+1,1\leq j\leq t-1\}.$ Note that $G_2=G_1+E_1-E_2.$ Then
\begin{eqnarray*}
&& x^{T} (\rho'-\rho)y\\
&=& x^{T}(A(G_2)-A(G_1))y\\
&=&\sum_{v_iz_j\in E_1}(x_{v_i}y_{z_j}+x_{z_j}y_{v_i})-\sum_{v_iw_j\in
E_2}(x_{v_i}y_{w_j}+x_{w_j}y_{v_i})\\
&=&r(t-1)(q-t)(x_2y_4+x_4y_1)-(r+1)(t-1)(x_2y_4+x_3y_2)\\
&=&(t-1)\left((r(q-t-1)-1)x_2y_4+r(q-t)x_4y_1-(r+1)x_3y_2\right)\\
&=&(t-1)x_3y_1\left((r(q-t-1)-1)\Big(\frac{t(p-r-1)}{\rho\rho'}+1\Big)+\frac{(r+1)(q-1)(p-r-1)}{\rho'^2}\right.\\
&&\left.-\frac{rt(q-t)(rt+1)}{\rho^2}\right)\triangleq (t-1)x_3y_1(*) .
\end{eqnarray*}
Suppose to the contrary that $\rho\geq\rho'.$ Then we have $\frac{1}{\rho'}\geq\frac{1}{\rho}.$ Note that $\rho'=\rho(D(p-r-1, r+1; 1,q-1))>\rho(K_{p-r-1,q})=\sqrt{(p-r-1)q}.$ Then $\rho>\sqrt{(p-r-1) q}.$ It follows that
\begin{eqnarray*}
(*)&\geq&(r(q-t-1)-1)\left(\frac{t(p-r-1)}{\rho^2}+1\right)+\frac{(r+1)(q-1)(p-r-1)}{\rho^2}-\frac{rt(q-t)(rt+1)}{\rho^2}\\
&\geq&(r(q-t-1)-1)\left(\frac{t(p-r-1)}{\rho^2}+\frac{(p-r-1)q}{\rho^2}\right)+\frac{(r+1)(q-1)(p-r-1)}{\rho^2}\\
&&-\frac{rt(q-t)(rt+1)}{\rho^2}\\
&\triangleq&\frac{g(t)}{\rho^2},
\end{eqnarray*}
where $g(t)=r^2t^3+[(1-q)r^2+(2-p)r]t^2+[r^2-(p+q-2)r-p+1]t+(1-q^2)r^2+(pq^2-q^2-p+2)r-p+1.$ Observe that $g'(t)=3r^2t^2+2[(1-q)r^2+(2-p)r]t+r^2+(2-p-q)r-p+1.$ A direct computation gives $g'(2)=(17-4q)r^2+(10-5p-q)r-p+1<0$ and
\begin{eqnarray*}
g'\left(\frac{p-2}{r}\right)&=&p^2+[(1-2q)r-5]p+r^2 +(3q-2)r+5\\
&\leq& q^2+[(1-2q)r-5]q+r^2+(3q-2)r+5\\
&=&(1-2r)q^2+(4r-5)q+r^2-2r+5\\
&<&0.
\end{eqnarray*}
This implies that $g(t)$ is decreasing for $2\leq t\leq\big\lfloor\frac{p-2}{r}\big\rfloor \leq \frac{p-2}{r}.$ Together with the condition $p\leq r(q-2)+1,$ we deduce that
\begin{eqnarray*}
g(t)&\geq& g\left(\frac{p-2}{r}\right)\\
&=&\frac{(-qr-1)p^2+(q^2r^2+3qr-r+3)p-q^2r^3-q^2r^2+r^3-2qr+ r-2}{r}\\
&\geq& \frac{q^2r^3+r(-4r^2+r+1)q+r(r^2-2r-2)}{r}\\
&>&0.
\end{eqnarray*}
Hence $x^{T} (\rho'-\rho)y>0,$ which implies that $\rho<\rho',$ a contradiction. Hence $\rho(D(p-rt-1,rt+1; t,q-t))<\rho(D(p-r-1, r+1; 1,q-1)).$
\end{proof}

Next we will prove a technical lemma which is crucial to our main result.

\begin{lem}\label{le8}
Let $r\geq1$ be an integer, and let $G\in \mathcal{B}_{p,q,r}.$ Each of the following holds.\\
(i) If $p\geq rq+1,$ then $\rho(G)\leq \sqrt{pq}$ with equality if and only if
$G\cong K_{p,q}.$\\
(ii) If $r(q-1)+2\leq p\leq rq,$ then $\rho(G)\leq \rho(D(p-r(q-1)-1,r(q-1)+1; q-1,1))$ with equality if and only if $G\cong D(p-r(q-1)-1,r(q-1)+1; q-1,1).$\\
(iii) If $p\leq r(q-1)+1,$ then $\rho(G)\leq \rho(D(p-r-1,r+1;1,q-1))$ with equality if and only if $G\cong D(p-r-1,r+1;1,q-1).$
\end{lem}
\begin{proof}

Let $B^*_{\rho}$ be the graph with the maximum spectral radius in $\mathcal{B}_{p,q,r}.$ Note that $b(B^*_{\rho})<\frac{1}{r}$ for integer $r\geq 1.$ By Lemma \ref{le4}, we  can select a binding set $S$ satisfying $S\subseteq X$ or $S\subseteq Y.$ Then $b(B^*_{\rho})=\frac{|N_{B^*_{\rho}}(S)|}{|S|} < \frac{1}{r}$ and $N_{B^*_{\rho}}(S)\neq V(B^*_{\rho}).$ Define $|S|=s$ and $|N_{B^*_{\rho}}(S)|=t.$ Then we can obtain that
\begin{eqnarray}\label{eq9}
s\geq rt+1.
\end{eqnarray}

\begin{claim}\label{claim12}
$B^*_{\rho}$ is connected.
\end{claim}
\begin{proof}

Suppose to the contrary that $B^*_{\rho}$ is disconnected. Let $B_1,$ $B_2,\dots ,$ $B_c$ are the connected components of $B^*_{\rho},$ where $B_i=(X_i, Y_i)$ for each $1\leq i\leq c.$ Hence there exists a connected component $B_i$ such that $S\cap V(B_i)\neq \emptyset.$ Without loss of generality, assume that $S\cap V(B_1)=S_1\neq \emptyset.$ We first consider the case $S\subseteq X.$ Note that $B^*_{\rho}$ contains no isolated vertices. Then $N_{B_1}(S_1)\neq \emptyset$ and $|X_i|\geq 1$ for $2\leq i\leq c.$ Define $B'= B^*_{\rho}+uv_2+uv_3+\cdots+uv_c,$ where $u\in V(N_{B_1}(S_1))$ and $v_i\in X_i$ for $2\leq i\leq c.$ One can check that $b(B')\leq \frac{|N_{B'}(S)|}{|S|}=\frac{|N_{B^*_{\rho}}(S)|}{|S|}<\frac{1}{r}.$ Then $B'\in \mathcal{B}_{p,q,r}.$ Note that $B'$ is connected and  $B^*_{\rho}$ is a proper subgraph of $B'.$ By Lemma \ref{le2}, we have $\rho(B')>\rho (B^*_{\rho}),$ which contradicts the maximality of $B^*_{\rho}.$  Hence $B^*_{\rho}$ is connected. From the above analysis, we can still conclude that $B^*_{\rho}$ is connected for $S\subseteq Y.$
\end{proof}

\begin{claim}\label{claim13}
If $S\subseteq X,$ then $B^*_{\rho}[(X/S)\cup Y]\cong K_{p-s,q}$ and $B^*_{\rho}[S\cup N_{B^*_{\rho}}(S)] \cong K_{s,t}.$
\end{claim}
\begin{proof}
Suppose to the contrary that $B^*_{\rho}[(X/S)\cup Y]\ncong K_{p-s,q}.$ Then there exist two vertices $u, v$ such that $uv\notin E(B^*_{\rho}),$ where $u\in X/S$ and $v\in Y.$ Let $B'= B^*_{\rho}+uv.$ Then $b(B')\leq \frac{|N_{B'}(S)|}{|S|}=\frac{|N_{B^*_{\rho}}(S)|}{|S|}<\frac{1}{r},$ which implies that $B'\in \mathcal{B}_{p,q,r}.$ By Claim \ref{claim12} and Lemma \ref{le2}, we have $\rho(B')>\rho (B^*_{\rho}),$ a contradiction. Hence $B^*_{\rho}[(X/S)\cup Y]\cong K_{p-s,q}.$ Similarly, $B^*_{\rho}[S\cup N_{B^*_{\rho}}(S)]\cong K_{s,t}.$
\end{proof}
By a similar argument to Claim \ref{claim13}, we obtain the following result.
\begin{claim}\label{claim13.0}
If $S\subseteq Y,$ then $B^*_{\rho}[X\cup (Y/S)]\cong K_{p,q-s}$ and $B^*_{\rho}[S\cup N_{B^*_{\rho}}(S)] \cong K_{t,s}.$
\end{claim}
By Claim \ref{claim13} and Claim \ref{claim13.0}, $B^*_{\rho}\cong D(p-s,s;t,q-t)$ for $S\subseteq X$ and $B^*_{\rho}\cong D(t,p-t;q-s,s)$ for $S\subseteq Y.$
\begin{claim}\label{claim14}
$s=rt+1.$
\end{claim}
\begin{proof}
We first consider $S\subseteq X.$ Note that $b(D(p-rt-1,rt+1;t,q-t))\leq \frac{t}{rt+1}<\frac{1}{r}.$ Then $D(p-rt-1,rt+1;t,q-t)\in \mathcal{B}_{p,q,r}.$ By (\ref{eq9}), $s\geq rt+1.$ If $s\geq rt+2,$ then $B^*_{\rho}$ is a proper subgraph of $D(p-rt-1,rt+1;t,q-t).$ By Lemma \ref{le2}, $\rho(D(p-rt-1,rt+1;t,q-t))>\rho(B^*_{\rho}),$ which contradicts the maximality of $B^*_{\rho}.$ Hence $s=rt+1.$ Following the same analysis, the result holds for $S \subseteq Y.$
\end{proof}
By Claim \ref{claim14}, $B^*_{\rho}\cong D(p-rt-1,rt+1;t,q-t)$ for $S\subseteq X$ and $B^*_{\rho}\cong D(t,p-t;q-rt-1,rt+1)$ for $S\subseteq Y.$
\begin{claim}\label{claim16}
If $S\subseteq Y$ and $p\leq rq,$ then $t=1.$
\end{claim}
\begin{proof}
Note that $S\subseteq Y.$ Then $s\leq q.$ By Claim \ref{claim14}, we have $t<s\leq q\leq p,$ and hence $t\leq p-1.$ If $s=q,$ then $B^*_{\rho}$ is disconnected, a contradiction. So $s\leq q-1.$ Since $s=rt+1,$ we have $1\leq t\leq\big\lfloor\frac{q-2}{r}\big\rfloor.$ Note that $A(B^*_{\rho})$ has an equitable quotient matrix
$$
R(A(B^*_{\rho}))=\left[
\begin{array}{cccc}
0 &0 &q-rt-1 &rt+1\\
0 &0 &q-rt-1 &0\\
t& p-t & 0 & 0\\
t & 0 & 0 & 0
\end{array}
\right].
$$
The characteristic polynomial of $R(A(B^*_{\rho}))$ is
$$g_1(x)=x^4+[(p-t)(rt+1)-pq]x^2+t(rt+1)(q-rt-1)(p-t).$$
Note that $A(D(1, p-1; q-r-1, r+1))$ has an equitable quotient matrix
$$
R(A(D(1, p-1; q-r-1, r+1)))=\left[
\begin{array}{cccc}
0 &0 &q-r-1 &r+1\\
0 &0 &q-r-1 &0\\
1& p-1 & 0 & 0\\
1 & 0 & 0 & 0
\end{array}
\right].
$$
The characteristic polynomial of $R(A(D(1, p-1; q-r-1, r+1)))$ is $$g_2(x)=x^4+((r+1)(p-1)-pq)x^2+(r+1)(q-r-1)(p-1).$$
Observe that $D(1, p-1; q-r-1,r+1)$ contains $K_{p,q-r-1}$ as a proper subgraph. Combining Lemma \ref{le2}, we have
\begin{eqnarray*}
\rho(D(1, p-1; q-r-1,r+1))>\rho(K_{p,q-r-1})=\sqrt{p(q-r-1)}.
\end{eqnarray*}
Let $g_3(x)=(r(p-t)-(r+1))x^2+r^2t^3-r(pr+q-r-2)t^2+((1-p)r^2+(q-2)(p-1)r-q+1)t+(r+1)(q-r-1)(p-1).$ Then we have $ g_1(x)-g_2(x)=(t-1)g_3(x).$ Suppose that $2\leq t\leq\big\lfloor\frac{q-2}{r}\big\rfloor.$ Then $r(p-t)-(r+1)\geq 0.$ The symmetry axis of $g_3(x)$ is $x=0<\sqrt{p(q-r-1)},$ which implies that $g_3(x)$ is increasing for $x\geq\sqrt{p(q-r-1)}.$ By calculation, we have
\begin{eqnarray*}
g_3(x)&\geq& g_3(\sqrt{p(q-r-1)})\geq\frac{(2r^3+3r^2+3r)(q-1)(r-1)}{r}>0.
\end{eqnarray*}
Then $\lambda_1(R(A(B^*_{\rho}))))<\lambda_1(R(A(D(1, p-1; q-r-1, r+1)))).$ According to Lemma \ref{le1}, we conclude that $\rho(B^*_{\rho})<\rho(D(1, p-1; q-r-1, r+1))$ for $2\leq t\leq\big\lfloor\frac{q-2}{r}\big\rfloor.$
Note that $D(1, p-1; q-r-1, r+1)\in\mathcal{B}_{p,q,r}.$ This contradicts the maximality of $B^*_{\rho}.$ Hence $t=1.$
\end{proof}

\begin{claim}\label{claim15}
If $S\subseteq X$ and $p\leq rq,$ then $1\leq t \leq \big\lfloor\frac{p-2}{r}\big\rfloor.$
\end{claim}
\begin{proof}
Since $S\subseteq X,$ we have $N_{B^*_{\rho}}(S)\subseteq Y,$ and hence $t\leq q.$ By $p\leq rq,$ Claim \ref{claim14} and $s\leq p,$ we have $1\leq t\leq \big\lfloor\frac{p-1}{r}\big\rfloor.$ If $q=t<s=p,$ then $b(B^*_{\rho})=\frac{|N_{B^*_{\rho}}(S)|}{|S|}=\frac{t}{s}=\frac{q}{p}\geq\frac{1}{r},$ a contradiction. If $t\leq q-1<s=p,$ then $B^*_{\rho}$ is disconnected, a contradiction. Hence $s\leq p-1.$ Since $s=rt+1,$ we have $1\leq t \leq \big\lfloor\frac{p-2}{r}\big\rfloor.$
\end{proof}

\begin{claim}\label{claim17}
$\rho(D(p-r-1, r+1; 1,q-1))\geq\rho(D(1, p-1; q-r-1, r+1))$ with equality if and only if $p=q.$
\end{claim}
\begin{proof}
Let $G_1=D(p-r-1, r+1; 1, q-1).$ Then $A(G_1)$ has an equitable quotient matrix
$$
R(A(G_1))=\left[
\begin{array}{cccc}
0 &0 &1 &q-1\\
0 &0 &1 &0\\
p-r-1& r+1 & 0 & 0\\
p-r-1 & 0 & 0 & 0
\end{array}
\right].
$$
The characteristic polynomial of $R(A(G_1))$ is
\begin{eqnarray*}
g_4(x)=x^4+((r+1)(q-1)-pq)x^2+(r+1)(q-1)(p-r-1).
\end{eqnarray*}
Recall that the characteristic polynomial of $R(A(D(1, p-1; q-r-1, r+1)))$ is
$$g_2(x)=x^4+((r+1)(p-1)-pq)x^2+(r+1)(p-1)(q-r-1).$$
Observe that $G_1$ contains $K_{r+1,1}$ as a proper subgraph. By Lemma \ref{le2}, we have $\rho(G_1)>\rho(K_{r+1,1})=\sqrt{r+1}.$ Since $g_2(x)-g_4(x)$ $=(r+1)(x^2-r)(p-q)\geq0$ for $x\geq \sqrt{r+1},$ $\lambda_1(G_1)\geq\lambda_1(R(A(D(1, p-1; q-r-1, r+1))))$ with equality if and only if $p=q.$ By Lemma \ref{le1}, we have $\rho(D(p-r-1, r+1; 1,q-1))\geq\rho(D(1, p-1; q-r-1, r+1))$ with equality if and only if $p=q.$
\end{proof}

(i) Since $p\geq rq+1,$ $K_{p,q}\in \mathcal{B}_{p,q,r}.$ Then $\rho(B^*_{\rho})\geq \rho(K_{p,q})=\sqrt{pq}.$
Note that $B^*_{\rho}$ is a proper subgraph of $K_{p,q}.$ Then $B^*_{\rho}\cong K_{p,q}.$

(ii) Note that $r(q-1)+2\leq p\leq rq.$ Then $r\geq2$ and $p>q.$ Recall that $D(p-r-1, r+1; 1,q-1)\in\mathcal{B}_{p,q,r}.$ If $S\subseteq Y,$ then
$B^*_{\rho}\cong D(t,p-t;q-rt-1,rt+1).$ By Claim \ref{claim16} and Claim \ref{claim17}, we obtain that
$\rho(D(p-r-1, r+1; 1,q-1))>\rho(D(1, p-1; q-r-1, r+1))=\rho(B^*_{\rho}),$ a contradiction. Hence $S\subseteq X$ and $B^*_{\rho}\cong D(p-rt-1,rt+1; t,q-t).$ Since $r(q-1)+2\leq p\leq rq$ and $r\geq2,$ $q-1=\big\lfloor\frac{p-2}{r}\big\rfloor.$ By Claim \ref{claim15}, we have $1\leq t\leq q-1.$ Observe that $A(B^*_{\rho})$ has an equitable matrix
$$
R(A(B^*_{\rho}))=\left[
\begin{array}{cccc}
0 &0 &t &q-t\\
0 &0 &t &0\\
p-rt-1& rt+1 & 0 & 0\\
p-rt-1 & 0 & 0 & 0
\end{array}
\right].
$$
The characteristic polynomial of $R(A(B^*_{\rho}))$ is given by
\begin{eqnarray*}
g_5(x)=x^4+((rt+1)(q-t)-pq)x^2+t(rt+1)(q-t)(p-rt-1).
\end{eqnarray*}
Let $G_2=D(p-r(q-1)-1, r(q-1)+1; q-1,1).$ Note that $A(G_2)$ has an equitable quotient matrix
$$
R(A(G_2))=\left[
\begin{array}{cccc}
0 &0 &q-1 &1\\
0 &0 &q-1 &0\\
p-r(q-1)-1& r(q-1)+1 & 0 & 0\\
p-r(q-1)-1 & 0 & 0 & 0
\end{array}
\right].
$$
The characteristic polynomial of $R(A(G_2))$ is $$g_6(x)=x^4+(r(q-1)-pq+1)x^2+(q-1)(qr-r+1)(p-rq+r-1).$$
Observe that $G_2$ contains $K_{p,q-1}$ as a proper subgraph. By Lemma \ref{le2}, we have
\begin{eqnarray*}
\rho(G_2)>\rho(K_{p,q-1})=\sqrt{p(q-1)}.
\end{eqnarray*}
Let $g_7(x)=(rt-r+1)x^2+(t(rt-r+1)-r(q-1)-1)p-r^2t^3+r(r-2)t^2+(qr^2-(r-1)^2)t+(q^2+1)r^2-2r(r-1)q-2r+1.$ Then $g_5(x)-g_6(x)=(q-t-1)g_7(x).$
The symmetry axis of $g_7(x)$ is $x=0<\sqrt{p(q-1)},$ which implies that $g_7(x)$ is increasing for $x\geq\sqrt{p(q-1)}.$ If $1\leq t\leq q-2,$ then combining $p\leq rq,$ we have
\begin{eqnarray*}
g_7(x)\geq g_7(\sqrt{p(q-1)})\geq (q-1)(qr^2-(p-2)r+p)\geq (q^2+q-2)r>0.
\end{eqnarray*}
Then $\lambda_1(R(A(B^*_{\rho})))<\lambda_1(R(A(G_2))).$ By Lemma \ref{le1}, we have
\begin{eqnarray*}
\rho(B^*_{\rho})<\rho(G_2),
\end{eqnarray*}
a contradiction. Hence $t=q-1,$ so $B^*_{\rho}\cong D(p-r(q-1)-1, r(q-1)+1; q-1,1).$

(iii) By Lemma \ref{le5}, $D(p-r-1,r+1;1,q-1)\in\mathcal{B}_{p,q,r}.$ Then
\begin{eqnarray}\label{eq16}
\rho(B^*_{\rho})\geq \rho(D(p-r-1,r+1;1,q-1)).
\end{eqnarray}
Now we divide our argument into the following two cases.

\vspace{1.5mm}
\noindent\textbf{Case 1.} $S\subseteq X.$
\vspace{1mm}

By $p\leq r(q-1)+1\leq rq$ and Claim \ref{claim15}, we have $1\leq t\leq\big\lfloor\frac{p-2}{r}\big\rfloor\leq q-2.$ Note that $B^*_{\rho}\cong D(p-rt-1,rt+1; t,q-t).$ Recall that the characteristic polynomial of $R(A(B^*_{\rho}))$ is
$$g_5(x)=x^4+((rt+1)(q-t)-pq)x^2+t(rt+1)(q-t)(p-rt-1).$$
Since $D(p-r-1, r+1; 1,q-1)$ contains $K_{p-r-1,q}$ as a proper subgraph, we have
\begin{eqnarray*}
\rho(D(p-r-1, r+1; 1,q-1))>\rho(K_{p-r-1,q})=\sqrt{(p-r-1)q}.
\end{eqnarray*}
Recall that the characteristic polynomial of $R(A(D(p-r-1, r+1; 1, q-1)))$ is
\begin{eqnarray*}
g_4(x)=x^4+((r+1)(q-1)-pq)x^2+(r+1)(q-1)(p-r-1).
\end{eqnarray*}
Let $g_8(x)=((q-t-1)r-1)x^2+(t^3-(t^2+t+1) (q-1))r^2+((t+1)(q-t)-1)(p-2)r+(q-t-1)(p-1).$ Then $g_5(x)-g_4(x)=(t-1)g_8(x).$ The symmetry axis of $g_8(x)$ is $x=0<\sqrt{(p-r-1)q},$ which implies that $g_8(x)$ is increasing with respect to $x\geq\sqrt{(p-r-1)q}$.

Next we claim that $t=1.$ Suppose that $2\leq t\leq\big\lfloor\frac{p-2}{r}\big\rfloor.$ If $2\leq t\leq\big\lfloor\frac{p-2}{r}\big\rfloor-1,$ then
\begin{eqnarray*}
g_8(x)&\geq& g_8(\sqrt{(p-r-1)q})\\
&\geq&\frac{(p-r-1)((-qr-r-1)p +q^2r^2+r^2q+2rq+2r+2)}{r}\\
&\geq&\frac{(p-r-1)(1+(q+1)r^2+2r)}{r}\\
&>&0,
\end{eqnarray*}
implying $\lambda_1(R(A(B^*_{\rho})))<\lambda_1(R(A(D(p-r-1, r+1; 1, q-1)))).$
By Lemma \ref{le1}, we have $\rho(B^*_{\rho})<\rho(D(p-r-1, r+1; 1, q-1)),$ a contradiction. Hence $t=\big\lfloor\frac{p-2}{r}\big\rfloor\leq q-2.$ If $t=\big\lfloor\frac{p-2}{r}\big\rfloor=q-2,$ then $r(q-2)+2\leq p\leq r(q-1)+1,$ and
\begin{eqnarray*}
g_8(x)\geq g_8(\sqrt{(p-r-1)q})\geq (r-1)((q^2-4q+1)r+q-1)>0.
\end{eqnarray*}
Hence $\rho(B^*_{\rho})<\rho(D(p-r-1, r+1; 1, q-1)),$ a contradiction. If $t=\big\lfloor\frac{p-2}{r}\big\rfloor\leq q-3,$ then by Lemma \ref{le7.0}, we have $\rho(B^*_{\rho})<\rho(D(p-r-1, r+1; 1, q-1)),$ a contradiction. Hence $t=1$ and $B^*_{\rho}\cong D(p-r-1, r+1; 1,q-1).$

\vspace{1.5mm}
\noindent\textbf{Case 2.} $S\subseteq Y.$
\vspace{1mm}

Note that $B^*_{\rho}\cong D(t,p-t;q-rt-1,rt+1)$ and $q\leq p\leq r(q-1)+1\leq rq.$ Combining (\ref{eq16}), Claim \ref{claim16} and Claim \ref{claim17}, we have $\rho(D(p-r-1, r+1; 1,q-1))\leq\rho(B^*_{\rho})=\rho (D(1, p-1; q-r-1, r+1))\leq\rho(D(p-r-1, r+1; 1,q-1)).$ Then $p=q$ and $B^*_{\rho}\cong D(p-r-1,r+1; 1,q-1).$
\end{proof}

\begin{lem}\label{le12}
If $n-q\leq r(q-1)+1$ and $r\geq2,$ then $\rho(D(n-q-r-1,r+1;1,q-1))\leq\rho\left(D\left(\big\lceil\frac{n-r-1}{2}\big\rceil,r+1;1,
\big\lfloor\frac{n-r-1}{2}\big\rfloor-1\right)\right)$ with equality if and only if $q=\big\lfloor\frac{n-r-1}{2}\big\rfloor.$
\end{lem}

\begin{proof}
Since $n-q\leq r(q-1)+1$ and $q\leq\lfloor\frac{n}{2}\rfloor,$ $\big\lfloor\frac{n+r-1}{r+1}\big\rfloor\leq q\leq\lfloor\frac{n}{2}\rfloor.$
Let $G_1=D(n-q-r-1,r+1;1,q-1).$ For short, let $\beta=\big\lfloor\frac{n-r-1}{2}\big\rfloor.$ Then $\big\lfloor\frac{n+r-1}{r+1}\big\rfloor\leq \beta\leq\lfloor\frac{n}{2}\rfloor.$
If $q=\beta,$ then $\rho(G_1)=\rho(D(n-\beta-r-1,r+1;1,\beta-1)).$ Next we will prove that $\rho(G_1)<\rho(D(n-\beta-r-1,r+1;1,\beta-1))$ for $q\neq \beta.$
Note that $A(G_1)$ has an equitable quotient matrix
$$
R(A(G_1))=\left[
\begin{array}{cccc}
0 &0 &1 &q-1\\
0 &0 &1 &0\\
n-q-r-1& r+1 & 0 & 0\\
n-q-r-1& 0 & 0 & 0
\end{array}
\right].
$$
The characteristic polynomial of $R(A(G_1))$ is
$$h_1(x)=x^4+(q^2-nq +rq+q-r-1)x^2+(r+1)(q-1)(n-q-r-1).$$
Let $G_2=D(n-\beta-r-1,r+1; 1,\beta-1).$ Then $A(G_2)$ has an equitable quotient  matrix
$$
R(A(G_2))=\left[
\begin{array}{cccc}
0 &0 &1 &\beta-1\\
0 &0 &1 &0\\
n-\beta-r-1& r+1 & 0 & 0\\
n-\beta-r-1& 0 & 0 & 0
\end{array}
\right].
$$
The characteristic polynomial of $R(A(G_2))$ is
$$h_2(x)=x^4+(\beta^2-n\beta+r\beta+\beta-r-1)x^2+(r+1)(\beta-1)(n-\beta-r-1).$$
Note that $G_2$ contains $K_{n-\beta-r-1,\beta}$ as a
proper subgraph. Combining Lemma \ref{le2}, we have
\begin{eqnarray*}
\rho(G_2)>\rho(K_{n-\beta-r-1,\beta})=\sqrt{(n-\beta-r-1)\beta}.
\end{eqnarray*}
Let
$h_3(x)=h_1(x)-h_2(x)=(\beta-q)((n-r-\beta-q-1)x^2+r^2+(\beta+q+1-n)r+\beta+q-n).$

If $q\leq \beta-1,$ then $\beta-q\geq1$ and $n-r-\beta-q-1>0,$ which implies that $h_3(x)$ is increasing on $x\geq
\sqrt{(n-\beta-r-1)\beta}$. Combining $n\geq r^2+r+2$ and $\frac{n-r-2}{2}\leq\beta\leq \frac{n-r-1}{2}$, we obtain that $h_3(x)\geq h_3(\sqrt{(n-\beta-r-1)\beta})\geq \frac{r^4}{4}+\frac{r^2}{2}-2r-\frac{7}{4}>0$. By Lemma \ref{le1}, we have $\rho(G_1)<\rho\left(D\left(\big\lceil\frac{n-r-1}{2}\big\rceil,r+1;1, \big\lfloor\frac{n-r-1}{2}\big\rfloor-1\right)\right).$

If $q\geq \beta+1$, then $\beta-q<0$ and $n-r-\beta-q-1\leq0$. It follows that $h_3(x)$ is increasing for $x\geq \sqrt{(n-\beta-r-1)\beta}$. By $n\geq r^2+r+2$ and $\frac{n-r-2}{2}\leq\beta\leq \frac{n-r-1}{2}$, we deduce that
$h_3(x)\geq h_3(\sqrt{(n-\beta-r-1)\beta})\geq r+1>0$. By Lemma \ref{le1}, we have $\rho(G_1)<\rho\left(D\left(\big\lceil\frac{n-r-1}{2}\big\rceil,r+1;1, \big\lfloor\frac{n-r-1}{2}\big\rfloor-1\right)\right).$
\end{proof}

Recall that $\rho'=\rho\left(K_{n-\big\lfloor\frac{n-1}{r+1}\big\rfloor, \big\lfloor\frac{n-1}{r+1}\big\rfloor}\right)$ and
$\rho''=\rho\left(D\left(\big\lceil\frac{n-r-1}{2}\big\rceil,r+1;1, \big\lfloor\frac{n-r-1}{2}\big\rfloor-1\right)\right).$

\medskip
\noindent  \textbf{Proof of Theorem \ref{thm7}.}
Let $\tilde{B^*_{\rho}}=(X,Y)$  be the bipartite graph with the maximum spectral radius in $\mathcal{B}_{n,r}.$ Without loss of generality, we assume that $|X|\geq |Y|.$ Let $|Y|=q.$ Then $|X|=n-q$ and $q\leq \big\lfloor\frac{n}{2}\big\rfloor.$

(i) Note that $r=1.$ We claim that $|X|=n-q\geq q+1.$ In fact, if $n-q\leq q,$ then $n=2q,$ which implies that $n$ is even. By Lemma \ref{le8} (iii), we have $\tilde{B^*_{\rho}}\cong D(\frac{n}{2}-2,2;1,\frac{n}{2}-1).$ Note that $K_{n-\big\lfloor\frac{n-1}{2}\big\rfloor,\big\lfloor\frac{n-1}{2}\big\rfloor}\in \mathcal{B}_{n,1}.$ Combining Lemma \ref{le7}, we obtain that $$\rho\big(\tilde{B^*_{\rho}}\big)=\rho\Big( D\Big(\frac{n}{2}-2,2;1,\frac{n}{2}-1\Big)\Big)\leq \sqrt{e\Big(D\Big(\frac{n}{2}-2,2;1,\frac{n}{2}-1\Big)\Big)}<\rho(K_{n-\big\lfloor\frac{n-1}{2}\big\rfloor,\big\lfloor\frac{n-1}{2}\big\rfloor}),$$
which contradicts the maximality of $\tilde{B^*_{\rho}}.$ Hence $n-q\geq q+1,$ and $q\leq\big\lfloor\frac{n-1}{2}\big\rfloor.$ By Lemma \ref{le8} (i), we have $\tilde{B^*_{\rho}}\cong K_{n-q,q}.$ Since $q\leq\big\lfloor\frac{n-1}{2}\big\rfloor,$ $\tilde{B^*_{\rho}}\cong K_{n-\big\lfloor\frac{n-1}{2}\big\rfloor,\big\lfloor\frac{n-1}{2}\big\rfloor}.$

(ii) Note that $r\geq2.$ We first prove that $n-q\geq rq+1.$ Suppose to the contrary that $n-q\leq rq.$ If $r(q-1)+2\leq n-q\leq rq,$ then $\left\lceil\frac{n}{r+1}\right\rceil\leq q\leq\big\lfloor\frac{n+r-2}{r+1}\big\rfloor.$ By Lemma \ref{le8} (ii), we have $\tilde{B^*_{\rho}}\cong D(n-(r+1)q+r-1,r(q-1)+1; q-1,1).$ Note that $K_{n-\big\lfloor\frac{n-1}{r+1}\big\rfloor,
\big\lfloor\frac{n-1}{r+1}\big\rfloor}\in \mathcal{B}_{n,r}.$ Combining Lemma \ref{le7} and $n\geq r^2+r+2,$ $$\rho(\tilde{B^*_{\rho}})\leq \sqrt{(n-q)q-r(q-1)-1}\leq \sqrt{\Big\lfloor\frac{n+r-2}{r+1}\Big\rfloor\Big(n-r-\Big\lfloor\frac{n+r-2}{r+1}\Big\rfloor\Big)+r-1}<\rho',$$ which contradicts the maximality of $\tilde{B^*_{\rho}}.$  If $n-q\leq r(q-1)+1,$ then by Lemma \ref{le8} (iii), we have $\tilde{B^*_{\rho}}\cong D(n-q-r-1,r+1;1,q-1).$ By $\rho'>\rho''$ and Lemma \ref{le12}, we have
\begin{eqnarray*}
\rho(\tilde{B^*_{\rho}})=\rho(D(n-q-r-1,r+1;1,q-1))\leq \rho''<\rho',
\end{eqnarray*}
which contradicts the maximality of $\tilde{B^*_{\rho}}.$ Hence $n-q\geq rq+1,$ and $1\leq q\leq\big\lfloor\frac{n-1}{r+1}\big\rfloor.$ By Lemma \ref{le8} (i), we have $\tilde{B^*_{\rho}}\cong K_{n-q,q}.$
Since $q\leq\big\lfloor\frac{n-1}{r+1}\big\rfloor,$ $\tilde{B^*_{\rho}}\cong K_{n-\big\lfloor\frac{n-1}{r+1}\big\rfloor,\big\lfloor\frac{n-1}{r+1}\big\rfloor}.$

(iii) Note that $r\geq2.$ We claim that $n-q\leq r(q-1)+1.$ By contradiction, assume that $n-q\geq r(q-1)+2.$ If $r(q-1)+2\leq n-q\leq rq,$ then by Lemma \ref{le8} (ii), we have $\tilde{B^*_{\rho}}\cong D(n-(r+1)q+r-1,r(q-1)+1; q-1,1).$ Note that
$D\big(\big\lceil\frac{n-r-1}{2}\big\rceil,r+1;1,
\big\lfloor\frac{n-r-1}{2}\big\rfloor-1\big)\in \mathcal{B}_{n,r}.$ Combining Lemma \ref{le7}, $n\geq r^2+r+2$ and $\rho'<\rho'',$ we can obtain that $$\rho(\tilde{B^*_{\rho}})\leq \sqrt{(n-q)q-r(q-1)-1}\leq \sqrt{\Big\lfloor\frac{n+r-2}{r+1}\Big\rfloor\Big(n-r-\Big\lfloor\frac{n+r-2}{r+1}\Big\rfloor\Big)+r-1}<\rho'<\rho'',$$
which contradicts the maximality of $\tilde{B^*_{\rho}}.$ If $n-q\geq rq+1,$ then $1\leq q\leq\big\lfloor\frac{n-1}{r+1}\big\rfloor.$ By Lemma \ref{le8} (i), we have $\tilde{B^*_{\rho}}\cong K_{n-q,q}.$ By $\rho'<\rho''$ and $q\leq\big\lfloor\frac{n-1}{r+1}\big\rfloor,$ we have $$\rho(\tilde{B^*_{\rho}})\leq \rho(K_{n-\big\lfloor\frac{n-1}{r+1}\big\rfloor,\big\lfloor\frac{n-1}{r+1}\big\rfloor})=\rho'<\rho'',$$ which contradicts the maximality of $\tilde{B^*_{\rho}}.$ Hence $n-q\leq r(q-1)+1.$ By Lemma \ref{le8} (iii), we have $\tilde{B^*_{\rho}}\cong D(n-q-r-1,r+1;1,q-1).$ If $q\neq\big\lfloor\frac{n-r-1}{2}\big\rfloor,$ then by Lemma \ref{le12}, we have $\rho(\tilde{B^*_{\rho}})<\rho'',$ a contradiction. Hence $q=\big\lfloor\frac{n-r-1}{2}\big\rfloor,$ which implies that $\tilde{B^*_{\rho}}\cong D\left(\big\lceil\frac{n-r-1}{2}\big\rceil,r+1;1, \big\lfloor\frac{n-r-1}{2}\big\rfloor-1\right).$

(iv) Based on the proofs (ii) and (iii), (iv) directly follows.\hspace*{\fill}$\Box$

\section{Concluding remarks}

The bipartite binding number was introduced in \cite{qian2001, Hu2013} to characterize structural properties of bipartite graphs. For a bipartite
graph $G=(X,Y),$ its {\it bipartite binding
number} $b^B(G)$ is defined as follows. If $G=K_{|X|,|Y|},$ then
$b^B(G)=\min\{|X|,|Y|\}.$ Otherwise,
$$b^B(G)=\min\left\{\min_{\substack{\emptyset\neq S\subseteq X \\ N_G(S) \subsetneq
Y}} \frac{|N_G(S)|}{|S|}, \min_{\substack{\emptyset\neq T\subseteq Y \\ N_G(T)
\subsetneq X}}
\frac{|N_G(T)|}{|T|}\right\}.$$
Hao et al.\cite{Hao2025} characterized the extremal graphs with the maximum size (spectral radius) among all connected balanced bipartite graphs with $b^B(G)<r$, where $r\geq1$ is an integer.

\begin{thm}[Hao et al.\cite{Hao2025}]\label{thm10}
Let $r$ be a positive integer, and let $G$ be a connected balanced bipartite graph of order $2n$ such that $b^B(G)<r.$ Each of the following holds.\\
(i) If $r=1$ and $n\geq3,$ then $e(G)\leq n^2-2n+2$ with equality if and only if $G\cong D(n-2, 2; 1, n-1).$\\
(ii) If $r\geq2,$ $n\geq r+1$ and $n\equiv 0 ~(\rm{mod}$$~r),$ then $e(G)\leq n^2-\frac{n}{r}$ with equality if and only if $G\cong D(n-\frac{n}{r}, \frac{n}{r}; n-1, 1).$\\
(iii) If $r\geq2,$ $n\geq r+1$ and $n\not\equiv 0 ~(\rm{mod}$$~ r)$, then $e(G)\leq n^2-\big\lceil\frac{n+1}{r}\big\rceil$ with equality if and only if $G\cong D(n-\big\lceil\frac{n+1}{r}\big\rceil,\big\lceil\frac{n+1}{r}\big\rceil; n-1, 1).$
\end{thm}

\begin{thm}[Hao et al.\cite{Hao2025}]\label{thm11}
Let $r$ be a positive integer, and let $G$ be a connected balanced bipartite graph of order $2n$ such that $b^B(G)<r.$ Each of the following holds.\\
(i) If $r=1$ and $n\geq3,$ then $\rho(G)\leq\rho(D(n-2, 2; 1, n-1))$ with equality if and only if $G\cong D(n-2, 2; 1, n-1).$\\
(ii) If $r\geq2,$ $n\geq r+1$ and $n\equiv 0 ~(\rm{mod}$$~r),$ then $\rho(G)\leq\rho(D(n-\frac{n}{r},\frac{n}{r}; n-1, 1))$ with equality if and only if $G\cong D(n-\frac{n}{r}, \frac{n}{r}; n-1, 1).$\\
(iii) If $r\geq2,$ $n\geq r+1$ and $n\not\equiv 0 ~(\rm{mod}$$~r),$ then $\rho(G)\leq\rho(D(n-\big\lceil\frac{n+1}{r}\big\rceil, \big\lceil\frac{n+1}{r}\big\rceil; n-1, 1))$ with equality if and only if $G\cong D(n-\big\lceil\frac{n+1}{r}\big\rceil, \big\lceil\frac{n+1}{r}\big\rceil; n-1, 1).$
\end{thm}

With the above characterization on connected balanced bipartite graphs, one may ask the following problem: What is the maximum size (spectral radius) for all bipartite graphs satisfying $b^B(G)<r,$ where $r\geq1$ is an integer? The methods and techniques developed in our work offer a viable and insightful approach to tackling this problem.

\vspace{5mm}
\noindent
{\bf Declaration of competing interest}
\vspace{3mm}

The authors declare that they have no known competing financial interests or personal relationships that could have appeared to influence the work reported in this paper.

\vspace{5mm}
\noindent
{\bf Data availability}
\vspace{3mm}

No data was used for the research described in this paper.

%

\end{document}